\documentclass[a4paper,11pt,reqno,oneside]{amsart}

\usepackage{csquotes}
\usepackage[backend=bibtex, giveninits, style=alphabetic, uniquelist=false, doi=false, url=false, isbn = false]{biblatex}
\usepackage{mk_paper}
\usepackage{stylefiles/abbrev_conc_ppp}

\title[Concentration inequalities for Poisson processes]{Concentration inequalities for Poisson point processes with application to adaptive intensity estimation}
\author{Martin Kroll}
\date{\normalsize \today} 

\begin{document}
	
\maketitle

\begin{abstract}
\noindent We derive concentration inequalities for maxima of empirical processes associated with Poisson point processes. 
The proofs are based on a careful application of Ledoux's entropy method.
We demonstrate the utility of the obtained concentration inequalities by application to adaptive intensity estimation.
\end{abstract}

\keywords{Poisson point process, concentration inequality, entropy method, adaptive estimation}

\vspace{4pt}

\normalsize


\section{Introduction}

Poisson point processes (PPPs) are of fundamental importance in probability theory and statistics, both from a theoretical and an applied point of view.
For instance, they serve as the elementary building blocks for complex point process models which are used in stochastic geometry~\cite{stoyan2013stochastic} and a wide range of applications including, amongst others, extreme value theory~\cite{resnick1987extreme}, finance~\cite{mikosch2009modelling}, forestry~\cite{penttinen2000recent}, and queueing theory~\cite{bremaud1981point}.

This paper is divided into two parts.
In the first one, inspired by results in~\cite{klein2005concentration}, we derive concentration inequalities for maxima of empirical processes associated with a PPP.
In the second part, we demonstrate the potential applicability of our results to statistics:
assuming that we observe i.i.d. PPPs with absolutely continuous intensity measure $\Lambda$, we consider the non-parametric estimation of the corresponding intensity function.
We derive optimal rates of convergence in terms of the sample size $n$ and propose a fully-data driven estimator of the intensity.
The theoretical study of the estimator is essentially based on the concentration inequalities derived in the first part.

Concentration inequalities provide upper bounds on the probability that a random variable deviates from its mean or median by a certain amount:
the well-known inequalities by Markov, Hoeffding and Bernstein provide classical examples of such inequalities.
Modern research on concentration inequalities goes back at least to the 1970s when concentration inequalities for deviations from the mean and median for Lipschitz continuous functions of multivariate Gaussian random variables were derived:
the papers~\cite{borell1975brunn} and~\cite{sudakov1978extremal} consider deviations from the median, whereas~\cite{cirelson1976norms} deals with deviations from the mean.
The main advantage of these results is that they do not depend on the dimension of the underlying Euclidean space and allow to control suprema of Gaussian processes.

In a series of papers, Talagrand~\cite{talagrand1995concentration, talagrand1996new} developed concentration inequalities for suprema of empirical processes associated with random variables on general product spaces.
His original proof of the \emph{Talagrand inequality} (see Theorem~1 in~\cite{massart2000about}, for instance) is rather technical and essentially based on geometric arguments. 
Ledoux~\cite{ledoux1996talagrands} proposed the \emph{entropy method} as an essentially different and more accessible approach to prove Talagrand's results.
However, the paper~\cite{ledoux1996talagrands} regained Talagrand's result with different numerical constants and variance factor only.
Massart's~\cite{massart2000about} approach is also based on Ledoux's method and a careful adaption of \emph{Gross's logarithmic Sobolev inequality} to a non-Gaussian framework and led to reasonable numerical constants in Talagrand type inequalities when the underlying random variables are independent but eventually non-identically distributed.
The paper~\cite{bousquet2003concentration} proves Talagrand type inequalities for sub-additive functions using the entropy method.
Finally, in the same setup as in~\cite{massart2000about}, Klein and Rio \cite{klein2005concentration} obtained concentration results by an elaborate application of the entropy method.
In addition, they provided a brief comparison of the numerical constants occurring in different papers on the subject showing that their results are optimal in some sense.
The main motivation of the first part of this paper is to transfer the results from \cite{klein2005concentration} to the setup with PPPs.

There is already some amount of research papers dealing with concentration inequalities for point processes, see for instance~\cite{bobkov1998modified, houdre2002remarks, houdre2002concentration} and~\cite{wu2000new}.
As pointed out in~\cite{reynaud2003adaptive} (cf.~p.~109 therein), the main drawback of all these results is that they provide a variance term that is difficult to deal with in statistical applications.
Recent concentration results have also been motivated by applications in stochastic geometry and  \cite{bachmann2016concentration,bachmann2016componentcounts}.

The scope of the paper~\cite{reynaud2003adaptive}, however, is similar to the one of our paper:
it is inspired by Talagrand type results from~\cite{massart2000about} and derives analogous results in a model with PPPs with exactly the same numerical constants as in the non-PPP setup. 
Similarly, the main intention of the present paper is to transfer the concentration inequalities from~\cite{klein2005concentration} to a PPP framework (again, by keeping exactly the same numerical constants as in~\cite{klein2005concentration}; see Remark~\ref{remCompReynaud} for a direct comparison of our result with the one of \cite{reynaud2003adaptive}).

Let us emphasize that it is not possible to apply the results from~\cite{klein2005concentration} immediately, that is, interpreting point processes as random variables in the space of locally finite measures, in order to obtain our results.
The stochastic integral of a constant function with respect to a PPP provides an example that can be dealt with by our result but does not fit into the framework of~\cite{klein2005concentration} (see Remark~\ref{rem:counterexample} below).
Furthermore, we do currently not see how one could derive the results of our paper via a simple Poissonization argument (however, also in the work \cite{reynaud2003adaptive}, the proof of the concentration inequalities was not based on such an argument).
To obtain our results, one key argument of the proof is borrowed from~\cite{reynaud2003adaptive}: the \emph{infinite divisibility} of PPPs is exploited and the underlying probability space split into two parts: on the first one, the proofs from~\cite{klein2005concentration} can be mimicked to a great extent (making a careful adaptation of the many auxiliary results in~\cite{klein2005concentration} necessary), whereas the probability of the second one can be neglected asymptotically.
Based on our main results, we derive Proposition~\ref{prop:conc} in the appendix as an auxiliary concentration result which represents a key tool for our statistical application: non-parametric adaptive intensity estimation.

Intensity estimation in parametric and non-parametric models has been dealt with in a wide range of monographs and research papers.
For a general treatment of the subject, we refer to~\cite{karr1991point} as a general introduction to the statistics of point processes,~\cite{kutoyants1998statistical} for examples of intensity estimation in different parametric and nonparametric models, and~\cite{moller2004statistical} for estimation in general spatial models.
Early approaches to nonparametric intensity estimation include kernel~\cite{rudemo1982empirical, kutoyants1998statistical} and histogram estimators~\cite{rudemo1982empirical}. 
In addition, the paper~\cite{rudemo1982empirical} already discusses adaptive estimation of the intensity.
Baraud and Birg\'e~\cite{baraud2009estimating} consider a Hellinger type loss function and propose a histogram estimator for intensity estimation.
Other contributions focus on non-linear wavelet thresholding techniques, cf., for instance, the articles~\cite{kolaczyk1999wavelet, willett2007multiscale, reynaud2010near, sansonnet2014wavelet} and~\cite{bigot2013intensity}.
The paper~\cite{bigot2013intensity} proposes a non-linear hard thresholding estimator for intensity estimation from indirect observations.
Moreover, there exist other approaches to nonparametric intensity estimation in more specific models.
Let us mention the paper~\cite{gregoire2000convergence} that proposes a minimum complexity estimator in the Aalen model and~\cite{patil2004counting} that uses a wavelet approach to estimation in a multiplicative intensity model, without making a claim to be exhaustive.


In~\cite{reynaud2003adaptive}, the focus is on a model slightly different from the one we will consider: estimation of the intensity on the interval $[0,T]$ from only one observation and asymptotics for $T \to \infty$ are considered whereas we assume the availability of an independent sample of realizations of the point process on the fixed interval $[0,1]$.
Common ground of our approach and the one taken in~\cite{reynaud2003adaptive} is the use of \emph{projection estimators}.
Beyond that, the statistical methodology of the present paper is rather motivated by the procedure developed in~\cite{johannes2013adaptive} for circular deconvolution.
We derive a minimax lower bound under abstract smoothness conditions and propose a projection estimator that can attain this lower bound.
We note that the proof of the minimax lower bound given in~\cite{reynaud2003adaptive} does not hold for ellipsoids defined in terms of the trigonometric basis (cf.~the remark below Definition~4 and Proposition~3 in~\cite{reynaud2003adaptive}), whereas we derive such a minimax lower bound.
In return, the analysis in~\cite{reynaud2003adaptive} is rather general with an emphasis on wavelet methods and not tailored to the trigonometric basis which we exclusively consider in this paper.

As usual in nonparametric statistics, the performance of our proposed projection estimator crucially depends on the appropriate selection of a dimension parameter.
Since the optimal choice of this parameter depends on the unknown intensity and is thus unavailable in practise, we propose a fully-data driven selection of the dimension parameter leading to an adaptive estimator of the intensity.
Our approach is based on model selection by minimization of a penalized contrast function.
The combination of model selection techniques and concentration inequalities has attracted great attention in non-parametric statistics~\cite{barron1999risk},~\cite{comte2015estimation}.
In particular, a considerable amount of research has been devoted to Gaussian regression and density estimation frameworks. We refer to the monograph~\cite{massart2007concentration} for results and further references concerning these two frameworks.
Concentration inequalities have already been exploited in the context of non-parametric intensity estimation in the papers~\cite{reynaud2003adaptive} and~\cite{baraud2009estimating} mentioned above but their approaches are different from our one.
Taking a different point of view, the paper~\cite{birge2007model} introduces an approach to model selection via hypothesis testing in the framework of intensity estimation.
Following the model selection approach as in~\cite{barron1999risk}, we obtain an adaptive estimator that attains the optimal rate of convergence for intensities belonging either to some Sobolev space or some space of (generalized) analytic functions.  

The paper is organized as follows. 
In Section~\ref{s:conc}, we derive concentration inequalities for deviations from the mean for empirical processes associated with PPPs.
Section~\ref{s:adaptive} deals with nonparametric intensity estimation: Subsection~\ref{subs:framework} introduces the model, in Subsection~\ref{subs:minimax} we derive the minimax theory, and in Subsection~\ref{subs:adaptive} we study the adaptive estimator.
Most of the proofs deferred to the appendix.

\section{Concentration inequalities for Poisson point processes}\label{s:conc}

Let $N$ be a PPP with finite intensity measure $\Lambda$ on some Polish space $\X$.
We denote the underlying probability space with $(\Omega, \As, \P)$.
Let $\Sc$ be a countable class of measurable functions from $\X$ to $[-1,1]$.
For $s \in \Sc$, let us define
\begin{equation}\label{eq:I(s)}
	I(s) = \int_{\X} s(x) (\dd N(x) - \dd\Lambda(x)).
\end{equation}
The first aim of this paper is to establish concentration inequalities for $Z =\sup_{s \in \Sc} S_n(s)$.
In the spirit of~\cite{klein2005concentration}, concentration inequalities for right-hand side and left-hand side deviations are established separately.
As already sketched in the introduction, one key idea for the proofs of both theorems is borrowed from the article~\cite{reynaud2003adaptive}, namely to exploit the infinite divisibility of PPPs.
This property is used to break the proof for the point process case into handy pieces such that the proof from~\cite{klein2005concentration} can be mimicked.
However, adapting the variety of auxiliary results used in the proofs of~\cite{klein2005concentration} to our setup is non-trivial.
The analogues of these auxiliary results are collected in Subsections~\ref{subsubs:notation:right} and~\ref{subsec:aux:left} in the appendix.

Since the infinite divisibility is the essential ingredient for the proofs, it is not clear whether and if yes, how, concentration results as given in Theorems~\ref{thm:right} and~\ref{thm:left} can be transferred to point processes which are not infinitely divisible.
For Cox processes (which represent a natural generalization of PPPs), the proof presented here fails and the development of tools to deal with this case might be worth further investigation.
On the contrary, it should be possible to adapt the proof to frameworks with other infinitely divisible processes such as Poisson cluster processes (see~\cite{stoyan2013stochastic}, p.~151) under additional assumptions (for instance, that the number of daughter points is bounded).


The following theorem (the proof of which is given in Appendix~\ref{s:app:right}) provides concentration inequalities for right-hand side deviations of $Z$ from its mean.

\begin{thm}\label{thm:right}
	Let $N$ be a PPP on a Polish space $\X$ with \emph{finite} intensity measures $\Lambda$, and $\Sc$ be a countable class of measurable functions from $\X$ to $[-1,1]$.
	For $s \in \Sc$, define $I(s)$ as in~\eqref{eq:I(s)} and consider $Z = \sup_{s \in \Sc} I(s)$.
	Let $L_{Z}(t) = \log \E [\exp(tZ)]$ denote the logarithm of the moment-generating function of $Z$ and $V = \sup_{s \in \Sc} \var \left( I(s) \right)$. Then, for any non-negative $t$,
	\begin{enumerate}[label=\alph*)]
		\item\label{thm:right:a}
		$\begin{aligned}
		\qquad \qquad L_Z(t) \leq t \E Z + \frac t 2 \left(2 \E Z + V \right) ( \exp( (e^{2t}-1)/2 ) - 1 ).
		\end{aligned}$
	\end{enumerate}
	Setting $\upsilon = 2  \E Z + V$, we obtain that, for any non-negative $x$,
	\begin{enumerate}[label=\alph*)]
		\setcounter{enumi}{1}
		\item
		$\begin{aligned}
		\qquad \qquad \P \left( Z \geq \E Z + x \right) \leq \exp \left(- \frac{x} {4} \log (1+2\log(1+x/\upsilon)) \right),
		\end{aligned}$
		\end{enumerate}
		and,
		\begin{enumerate}[label=\alph*)]
		\setcounter{enumi}{2}
		\item\label{thm:right:c} for any $x \geq 0$,
		\begin{align*}
			\P \left( Z \geq \E Z + x \right) &\leq \exp \left( - \frac{x^2}{\upsilon + \sqrt{\upsilon^2+ 3 \upsilon x} + (3x/2)} \right)\\
			&\leq \exp \left( - \frac{x^2}{2\upsilon + 3x}\right).
		\end{align*}
		\end{enumerate}
\end{thm}

\begin{rem}\label{rem:counterexample}
	We emphasize that Theorem~\ref{thm:right} cannot be immediately deduced from Theorem~1.1 in~\cite{klein2005concentration}.
	For instance, for $s \equiv 1$ the stochastic integral $\int_\X s(x)\dd N(x)$ is an unbounded function of $N$ but obviously $s\equiv 1$ fits into the framework of Theorem~\ref{thm:right}.
\end{rem}

\begin{rem}\label{remCompReynaud}
	Let us compare our result with the concentration inequality derived in \cite{reynaud2003adaptive}.
	More precisely, we consider Statement~\label{thm:right:c} from Theorem~\label{thm:right}, and compare it with Corollary~2 from \cite{reynaud2003adaptive} which is most closely related to our results.
	First note that Corollary~2 in \cite{reynaud2003adaptive} is formulated for the quantity $Z = \sup_{s \in \Sc} \vert I(s) \vert$ instead of $Z = \sup_{s \in \Sc}  I(s) $ (however, in our application we will consider a symmetric set $\Sc$ and the two definitions coincide).
	Then, Corollary~2 from \cite{reynaud2003adaptive}, reads as follows:
	\begin{align*}
		\P (Z \geq (1+ \epsilon) \E Z + x) \leq \exp \bigg( -\frac{x^2}{12 \upsilon_0 +2 \kappa(\epsilon)x } \bigg)
	\end{align*}
	where $\epsilon > 0$ is arbitrary but fixed, $\kappa(\epsilon) = 5/4 + 32/\epsilon$, and
	\begin{equation*}
		\upsilon_0 = \sup_{s \in \Sc} \int s^2(x) \dd \Lambda(x).
	\end{equation*}
\end{rem}

\begin{rem}
	In analogy to Corollary~1.1 in \cite{klein2005concentration}, the assumptions of Theorem~\ref{thm:right} imply that $\var Z \leq V + 2 \E Z$.
\end{rem}

The quantity $V$ is usually referred to as the \emph{wimpy variance} (cf.~\cite{boucheron2016concentration}, Chapter~11).
In the proof of~Proposition~\ref{prop:conc} in Appendix~\ref{s:integrated}, a suitable bound for the quantity $\upsilon$ will be determined. 

For the statistical applications we have in mind, we state the in the following an immediate corollary of Theorem~\ref{thm:right}.
For the formulation of the corollary, we consider $n$ independent PPPs on the Polish space $\X$ with finite intensity measures $\Lambda_1,\ldots,\Lambda_n$.
Again we denote the common underlying probability space of the PPPs with $(\Omega, \As, \P)$.
Let further $\Sc$ be a countable class of measurable functions from $\X$ to $[-1,1]^n$.
For $s=(s^1,\ldots,s^n) \in \Sc$ and $k \in \{ 1,\ldots,n \}$, let us define
\begin{equation}\label{eq:S_n(s)}
I^k(s) = \int_{\X} s^k(x) (\dd N_k(x) - \dd \Lambda_k(x)) \quad \text{ and } \quad S_n(s) = I^1(s) + \ldots + I^n(s).
\end{equation}

\begin{cor}\label{COR:RIGHT}
	Let $N_1,\ldots,N_n$ be independent PPPs on a Polish space $\X$ with \emph{finite} intensity measures $\Lambda_1,\ldots,\Lambda_n$, and $\Sc$ be a countable class of measurable functions from $\X$ to $[-1,1]^n$.
	For $s \in \Sc$, define $S_n(s)$ as in~\eqref{eq:S_n(s)} and consider $Z = \sup_{s \in \Sc} S_n(s)$.
	Let $L_{Z}(t) = \log \E [\exp(tZ)]$ denote the logarithm of the moment-generating function of $Z$ and $V_n = \sup_{s \in \Sc} \var \left( S_n(s) \right)$. Then, for any non-negative $t$,
	\begin{enumerate}[label=\alph*)]
		\item\label{cor:right:a}
		$\begin{aligned}
		\qquad \qquad L_Z(t) \leq t \E Z + \frac t 2 \left(2 \E Z + V_n \right) ( \exp( (e^{2t}-1)/2 ) - 1 ).
		\end{aligned}$
	\end{enumerate}
	Setting $\upsilon = 2  \E Z + V_n$, we obtain that, for any non-negative $x$,
	\begin{enumerate}[label=\alph*)]
		\setcounter{enumi}{1}
		\item
		$\begin{aligned}
		\qquad \qquad \P \left( Z \geq \E Z + x \right) \leq \exp \left(- \frac{x} {4} \log (1+2\log(1+x/\upsilon)) \right),
		\end{aligned}$
	\end{enumerate}
	and,
	\begin{enumerate}[label=\alph*)]
		\setcounter{enumi}{2}
		\item\label{cor:right:c} for any $x \geq 0$,
		\begin{align*}
		\P \left( Z \geq \E Z + x \right) &\leq \exp \left( - \frac{x^2}{\upsilon + \sqrt{\upsilon^2+ 3 \upsilon x} + (3x/2)} \right)\\
		&\leq \exp \left( - \frac{x^2}{2\upsilon + 3x}\right).
		\end{align*}
	\end{enumerate}
\end{cor}

The following theorem (the proof of which is given in Appendix~\ref{s:app:left}) provides concentration inequalities for left-hand side deviations of $Z$ from its mean.

\begin{thm}\label{thm:left}
	Under the assumptions of Theorem~\ref{thm:right}, for any non-negative $t$, it holds
	\begin{enumerate}[label=\alph*)]
		\item
		$\begin{aligned}
		\qquad \qquad L_Z(-t) \leq -t \E Z + \frac \upsilon 9 (e^{3t}- 3t -1).
		\end{aligned}$
	\end{enumerate}
	Consequently, for any non-negative $x$,
	\begin{enumerate}[label=\alph*)]
		\setcounter{enumi}{1}
		\item
		$\begin{aligned}
		\qquad \qquad \P \left( Z \leq \E Z - x \right) \leq \exp \left(- \frac{\upsilon} {9} h\left( \frac{3x}{\upsilon} \right)  \right),
		\end{aligned}$
	\end{enumerate}
	where $h(x) = (1+x) \log(1+x) - x$, and
	\begin{enumerate}[label=\alph*)]
		\setcounter{enumi}{2}
		\item
		for any $x \geq 0$,
		\begin{align*}
		\qquad \P \left( Z \leq \E Z - x \right) &\leq \exp \left( - \frac{x^2}{\upsilon + \sqrt{\upsilon^2+ 2 \upsilon x} + x} \right)\\
		&\leq \exp \left( - \frac{x^2}{2\upsilon + 2x}\right).
		\end{align*}
	\end{enumerate}
\end{thm}

\begin{rem}
	The concentration inequalities in Theorems~\ref{thm:right} and~~\ref{thm:left} translate literally (that is, with exact coincidence of the numerical constants) the ones obtained in~\cite{klein2005concentration} to our framework with PPPs.
	This observation is in line with the remark made in~\cite{reynaud2003adaptive} where the derived concentration inequalities translate literally previous results due to~\cite{massart2000about}.
\end{rem}

\begin{rem}
	In many situations of interest, it is possible to apply the concentration inequalities proved in this section (and the one proved in Appendix~\ref{s:integrated}) to \emph{non-countable} classes of measurable functions.
	A rigorous foundation of this practice can be based on density arguments (see Remarque 2.1 in~\cite{chagny2013estimation}).
\end{rem}

\section{Non-parametric intensity estimation}\label{s:adaptive}

\subsection{Model assumptions}\label{subs:framework}

Let $N_1,\ldots,N_n$ be i.i.d.\ realizations of a PPP on $[0,1]$ with square-integrable intensity function $\lambda \in \L^2 = \L^2([0,1], \dd x)$.
The $N_i$ can be interpreted as $\N_0$-valued random measures which motivates the notation $N_i=\sum_{j} \delta_{x_{ij}}$ (here, $\delta_x$ denotes the Dirac measure with mass concentrated at $x$).
Our aim is to estimate the intensity function $\lambda$ from the sample $N_1,\ldots,N_n$.
We consider the orthonormal basis $\{ \phi_j \}_{j \in \Z}$ of $\L^2$ which is given by $\phi_0 = 1$, and
\begin{equation*}
	\phi_j(t)=\sqrt 2 \cos(2\pi jt), \qquad \text{resp.} \qquad \phi_{-j}(t)=\sqrt 2 \sin(2\pi j t)
\end{equation*}
for $j = 1,2,\ldots$.
Define the sequence $(\beta_j)_{j \in \Z}$ of Fourier coefficients via $\beta_j = \int_0^1 \lambda(t) \phi_j(t) \dd t$ which yields the $\L^2$-convergent representation
\begin{equation}\label{eq:lambda:fou:ser}
	\lambda = \sum_{j \in \Z} \beta_j \phi_j.
\end{equation}

In order to evaluate the performance of an arbitrary estimator $\widetilde \lambda$ of $\lambda$, we consider the mean integrated squared error
$\E [\Vert \widetilde \lambda - \lambda \Vert^2]$ (where, as usual, the expectation is taken under the true intensity function $\lambda$ and $\Vert \cdot \Vert$ denotes the $\L^2$-norm).
We hold the minimax point of view and consider the \emph{maximum risk} defined by
$\sup_{\lambda \in \Lambda} \E [\Vert \widetilde \lambda - \lambda \Vert^2]$
for some smoothness class $\Lambda$ of potential intensity functions.
The corresponding \emph{minimax risk} is defined by
\begin{equation*}
	\inf_{\widetilde \lambda} \sup_{\lambda \in \Lambda} \E [\Vert \widetilde \lambda - \lambda \Vert^2]
\end{equation*}
where the infimum is taken over all potential estimators $\widetilde \lambda$ of $\lambda$ based on the sample $N_1,\ldots,N_n$.
An estimator $\lambda^\ast$ is called \emph{rate optimal} if
	$\sup_{\lambda \in \Lambda} \E [\Vert \lambda^\ast - \lambda \Vert^2] \lesssim \inf_{\widetilde \lambda} \sup_{\lambda \in \Lambda} \E [\Vert \widetilde \lambda - \lambda \Vert^2]$,
where the notation $a_n \lesssim b_n$ means that $a_n \leq C b_n$ for some numerical constant that does not depend on $n$. 
The specific form of the class $\Lambda$ will be introduced in the following Subsection~\ref{subs:minimax}.

For the moment, let us introduce the general type of projection estimator we will consider throughout this work:
Since $\E[ \int_{0}^{1} \phi_j(t)\dd N_i(t) ] = \beta_j$ for all $j \in \Z$ by Campbell's theorem (cf., for instance,~\cite{streit2010poisson}, Chapter~2),  $\widehat{\fou{\lambda}}_j = \frac{1}{n} \sum_{i=1}^{n} \int_{0}^{1} \phi_j(t)\dd N_i(t)$ is an unbiased estimator of $\beta_j$,
Equation~\eqref{eq:lambda:fou:ser} strongly suggests to consider orthogonal series estimators of the form
\begin{equation}\label{eq:def:ose}
	\widehat \lambda_k = \sum_{0 \leq \vert j \vert \leq k} \betahat_j \phi_j,
\end{equation}
where the dimension parameter $k \in \N_0$ has to be chosen appropriately.

\begin{rem}
	The estimator $\widehat \lambda_k$ is by definition not guaranteed to attain only non-negative values (which holds for the true intensity $\lambda$).
	In practise, this undesirable feature can avoided by considering the estimator $\widehat \lambda_{k+}$ defined via $\widehat \lambda_{k+}(t) = \widehat \lambda_k(t) \vee 0$, the risk of which is evidently bounded from above by the one of $\widehat \lambda_k$.
\end{rem}

\subsection{Minimax theory}\label{subs:minimax}

In order to define the class of admissible intensity functions in the definition of the minimax risk, let $\gamma = ( \gamma_j )_{j \in \Z}$ be a strictly positive symmetric sequence of weights and $r > 0$.
Set
\begin{equation*}
\Lambda = \Lambda(\gamma,L) = \{ \lambda \in \L^2 : \lambda \geq 0 \text{ and } \sum_{j \in \Z} \gamma_j^2 \beta_j^2 \eqdef \Vert \lambda \Vert_\gamma^2 \leq L^2 \}.
\end{equation*}
In the following, our aim is to study the minimax risk with respect to the function class $\Lambda$.
Our results will be obtained under the following mild regularity assumptions on the sequence $\gamma$.

\setcounter{ass}{0}
\begin{ass}\label{ass:seq}
	$\gamma = ( \gamma_j )_{j \in \Z}$ is a strictly positive symmetric sequence such that $\gamma_0=1$ and $(\gamma_n)_{n \in \N_0}$ is non-decreasing.
\end{ass}

The following proposition provides an upper risk bound for the estimator $\widehat \lambda_k$ defined in~\eqref{eq:def:ose} under an appropriate choice of the dimension parameter $k$.

\begin{prop}\label{prop:upper}
	Let Assumption~\ref{ass:seq} hold.
	Consider the estimator $\widehat \lambda_\knast$ with dimension parameter $\knast = \argmin_{k \in \N_0} \max \{ \gamma_k^{-2}, \frac{2k+1}{n} \}$.
	Then, for any $n \in \N$,
	\begin{equation*}
	\sup_{\lambda \in \Lambda} \E [ \Vert \widehat \lambda_\knast - \lambda \Vert^2] \lesssim \Psin \defeq \max\left\lbrace  \gamma_\knast^{-2},  \frac{2\knast+1}{n} \right\rbrace
	\end{equation*}
	where the constant hidden in $\lesssim$ depends only on $L$.
\end{prop}

\begin{proof}
	Introduce the function $\lambda_\knast \defeq \sum_{0 \leq \vert j \vert \leq \knast} \beta_j \phi_j$ which suggests the decomposition
	\begin{equation*}
	\E [\Vert \widehat{\lambda}_\knast - \lambda \Vert^2] = \Vert \lambda - \lambda_\knast \Vert^2 + \E [\Vert \widehat \lambda_\knast - \lambda_\knast \Vert^2]
	\end{equation*}
	of the considered risk into squared bias and variance.
	Using the smoothness assumption $\lambda \in \Lambda$, it is easy to see that $\Vert \lambda - \lambda_\knast \Vert^2 \leq L^2 \gamma_\knast^{-2}$ and $\E [\Vert \widehat \lambda_\knast - \lambda_\knast \Vert]^2 \leq L \cdot \frac{2\knast + 1}{n}$ and the statement of the theorem follows.
\end{proof}

The rate-optimality of the estimator $\widehat \lambda_\knast$ considered in Proposition~\ref{prop:upper} is demonstrated by means of the following theorem which is valid under mild additional assumptions.
The proof makes use of an adaptation of standard techniques in non-parametric statistics for the derivation of minimax lower bounds to our point process framework and is deferred to Appendix~\ref{APP:LOWER}.

\begin{thm}\label{thm:lower} 
	Let Assumption~\ref{ass:seq} hold and further assume that
	\begin{enumerate}[label=\color{blue}(C\arabic*), leftmargin=1.5cm]
		\item\label{it:C1} $\Gamma = \sum_{j \in \Z} \gamma_j^{-2} < \infty$, and
		\item $0 < \eta^{-1} = \inf_{n \in \N} (\Psin)^{-1} \min \{  \gamma_\knast^{-2},  \frac{2\knast+1}{n} \}$
		for some $\eta \geq 1$,
	\end{enumerate}
	where the quantities $\knast$ and $\Psin$ are defined in Proposition~\ref{prop:upper}.
	Then, for any $n \in \N$,
	\begin{equation*}
	\inf_{\widetilde \lambda} \sup_{\lambda \in \Lambda_\gamma^r} \E [ \Vert \widetilde \lambda - \lambda \Vert^2] \gtrsim \Psin
	\end{equation*}
	where the infimum is taken over all estimators of $\widetilde \lambda$ of $\lambda$ based on the sample $N_1,\ldots,N_n$ and the constant hidden in $\gtrsim$ depends only on $\eta$, $L$ and $\Gamma$.
\end{thm}

\begin{rem}
	The mild assumption~\ref{it:C1} on the convergence of the series $\sum_{j \in \Z} \gamma_j^{-2}$ is needed only in order to guarantee the non-negativity of the candidate intensities considered in the proof.
	On the whole, the proof is very much in line with the proof of Theorem~2.1 in~\cite{johannes2013adaptive} expanded with the essential ingredient that the Hellinger distance between two PPPs is bounded by the Hellinger distance of the corresponding intensity measures (see Theorem~3.2.1 in~\cite{reiss1993course}).
\end{rem}

\begin{rem}
	Note that the lower bound proof given in~\cite{reynaud2003adaptive} is based on a specific property called \emph{localization} and is not valid for ellipsoids expressed in terms of the trigonometric basis.
\end{rem}

\begin{ex}[Sobolev ellipsoids]\label{ex:rates}
	Let $\gamma_0=1$, $\gamma_j = \vert j \vert^{p}$ for $j \neq 0$.
	This setting corresponds to $\lambda$ belonging to a \emph{Sobolev ellipsoid}.
	Then, Assumption~\ref{ass:seq} is satisfied and elementary computations show that $\knast \asymp n^{1/(2p+1)}$ as well as $\Psin \asymp n^{-2p/(2p+1)}$.
	Furthermore, the additional conditions of Theorem~\ref{thm:lower} are satisfied if $p>1$ holds.
\end{ex}

\begin{ex}[Analytic functions]\label{ex:rates:analytic}
	Let $\gamma_j = \exp(\rho \vert j \vert)$ for $j \in \Z$ for some $\rho > 0$. This setting corresponds to $\lambda$ belonging to a class of \emph{analytic functions}.
	Assumption~\ref{ass:seq} is also fulfilled in this case and we obtain $\knast \asymp \log n$ and $\Psin \asymp \log n/n$.
	The additional assumption of Theorem~\ref{thm:lower} does not impose any additional restriction on $\rho$.
\end{ex}

\begin{ex}[Generalized analytic functions]\label{ex:rates:generalized}
	Let $\gamma_j = \exp(2\rho \vert j \vert^p)$ for $\beta, p > 0$.
	Note that in this case the Fourier coefficients of $\lambda$ obey a power exponential decay and $\lambda$ belongs to a class of \emph{generalized analytic functions}.
	Assumption~\ref{ass:seq} is satisfied in this case and there are no additional restrictions on $p$ (and $\rho$) due to Theorem~\ref{thm:lower}.
	We have $\knast \asymp (\log n)^{1/p}$ resulting in the rate $\Psin \asymp (\log n)^{1/p}/n$.
\end{ex}

\subsection{Adaptive estimation}\label{subs:adaptive}

The optimal choice of the dimension parameter $k$ stated in Proposition~\ref{prop:upper} depends on the smoothness characteristics of the intensity via the sequence $\gamma$.
However, such an \emph{a priori} knowledge is a strong assumption and usually not available in practise.
Thus, there is a demand for a data-driven choice of the dimension parameter which hopefully does not deteriorate the quality of the upper risk bound or at least leads to worse numerical constants, merely.

This data-driven choice of the dimension parameter and the resulting upper risk bound are investigated now.
For this purpose, we follow a model selection approach which has been successfully applied to a wide range of estimation problems in nonparametric statistics (cf., for instance,~\cite{barron1999risk, comte2015estimation} for general accounts to this model selection paradigm).

For $s,t \in \L^2$, introduce the notation $\langle s,t \rangle = \int_0^1 s(x) t(x) \dd x$ and consider the contrast function
\begin{equation*}
	\Upsilon_n(t) = \Vert t \Vert^2 - 2 \langle \widehat \lambda_n, t  \rangle, \qquad t \in \L^2.
\end{equation*}
Define the random sequence of penalties $(\pen_k)_{k \in \N}$ via $$\pen_k = 24 \cdot (\beta_0 \vee 1)\cdot \frac{2k+1}{n}.$$
Building on the definitions made until now, we define the data-driven selection $\knhat$ of the dimension parameter as the minimizer of the penalized contrast
\begin{equation*}\label{eq:def:k}
	\knhat \defeq \argmin_{0 \leq k \leq n} \{ \Upsilon_n(\widehat \lambda_k) + \pen_k \}.
\end{equation*}

The following theorem provides an upper bound for the risk of the estimator $\widehat \lambda_{\knhat}$.
Its proof is given in Appendix~\ref{s:app:adap}.

\begin{thm}\label{thm:adap:est}
	Let Assumption~\ref{ass:seq} hold.
	Then, for any $n \in \N$, we have
	\begin{equation*}
		\sup_{\lambda \in \Lambda_\gamma^r} \E [\Vert \widehat \lambda_\knhat - \lambda \Vert^2] \lesssim \min_{0 \leq k \leq n} \max \left\lbrace \gamma_k^{-2}, \frac{2k+1}{n} \right\rbrace + \frac{1}{n} + \exp(-\kappa \sqrt n)
	\end{equation*}
	where $\kappa > 0$ is a numerical constant and the constant hidden in $\lesssim$ depends only on $L$.
\end{thm}

\begin{rem}
	The penalty term used in the definition of $\knhat$ is random which is in contrast to penalty terms occuring, for instance, in density estimation or deconvolution problems.
	The need for randomization is due to the quantity $\beta_0$ in the definition of $H$ in Lemma~\ref{l:ex:conc}.
	If $L$ (but not $\gamma$) was known, one could proceed without randomization by choosing the penalty proportional to $\sqrt L (2k+1)/n$.
	However, the factor $L$ in this definition cannot be replaced by an estimate of $L$ because a reasonable estimator of $L$ is not reachable from the data.
	Note that the penalty terms considered in~\cite{reynaud2003adaptive} in a point process framework similar to ours are also non-deterministic.
\end{rem}

The adaptive estimator $\widehat \lambda_{\knhat}$ attains the rate $\Psin$ if and only if
\begin{equation*}
	\Psin \asymp \min_{0 \leq k \leq n} \max \bigg\{ \gamma_k^{-2}, \frac{2k+1}{n} \bigg\}.
\end{equation*}
Since under Assumption~\ref{ass:seq} it holds that $\knast \lesssim n$, we immediately obtain the following result.

\begin{cor}\label{cor}
	Under Assumption~\ref{ass:seq}, the estimator $\widehat \lambda_\knhat$ is rate optimal over the class $\Lambda$.
\end{cor}

In particular, the estimator $\widehat \lambda_\knhat$ is rate optimal in the framework of Examples~\ref{ex:rates}, \ref{ex:rates:analytic}, and~\ref{ex:rates:generalized} where $\knast \asymp n^{1/(2p+1)}$, $\knast \asymp \log n$, and $\knast \asymp (\log n)^{1/p}$, respectively.

%

%

\newpage

\appendix

\section{Proof of Theorem~\ref{thm:right}}\label{s:app:right}

\subsection{Notation and preparatory results}\label{subsubs:notation:right}

In this subsection, we introduce notation and state preliminary results.
The proof of Theorem~\ref{thm:right}, based on these results, is given in Subsection~\ref{subs:proof}. 
The key property used to prove Theorem~\ref{thm:right} is the \emph{infinite divisibility} of the PPP
$N$:
for any $\ell \in \N$, there exist i.i.d.\,PPPs $N_{j}$ such that
\begin{equation}\label{eq:inf:div}
N \stackrel{d}{=} \sum_{j=1}^{\ell} N_{j}.
\end{equation}
The common intensity measure of the $N_{j}$ in this representation is $\Lambdatilde = \Lambda/\ell$.
Throughout this work, the dependence of $N_{j}$, $\Lambdatilde$, and derived quantities on $\ell$ is often suppressed for the sake of convenience.
Define $\Lambdab = \Lambda(\X)$ and $\Delta = \Lambdab/\ell$.
For $s \in \Sc$, consider the centred random variables
\begin{equation*}
	I^{j}(s) = \int_\X s(x) (\dd N_{j}(x)-\dd \Lambdatilde(x)).
\end{equation*}
We define the random variable $X_{j} = N_{j}(\X)$ (that is, $X_{j}$ is the total number of points of the point process $N_{j}$) and the event $\Omega_{j} = \{ X_{j} \leq 1 \}$.

\begin{lem}\label{lem:prob}
	$\P (\Omega_{j}^\complement) \leq \Delta^2/2$.
\end{lem}

\begin{proof}
	The function $h: \N_0 \to \R, n \mapsto n^2 - n$ is non-negative and non-decreasing. Since $\Omega_{j}^\complement = \{ X_{j} \geq 2 \}$ the claim estimate follows from Markov's inequality. 
\end{proof}
Let us define the $\sigma$-fields
\begin{equation*}
	\Fs = \sigma (N_1,\ldots,N_\ell) \quad \text{and} \quad \Fs^{\vee j} = \sigma\left( \{N_{1},\ldots,N_{\ell} \}\backslash \{ N_{j}  \}\right).
\end{equation*}
Further, let $\E^{\vee j}[\,\cdot\, ] = \E [ \,\cdot\, | \Fs^{\vee j} ]$ and $\P^{\vee j}(A) = \E^{\vee j}[ \1_A ]$.
In addition, for the rest of Appendix~\ref{s:app:right}, we denote $f=f(t) = \exp(tZ)$ and $f_{j}=f_{j}(t) = \E^{\vee j} [f]$.
It will turn out to be sufficient to prove the results of this subsection under the following finiteness assumption.
\setcounter{ass}{5}
\begin{ass}\label{ass:finite}
	$\Sc = \{s_1,\ldots,s_m\}$ is a \emph{finite} set of measurable functions from the Polish space $\X$ to $[-1,1]$, and
	$\tau$ is the first index $i$ such that $Z = I(s_i)$.
\end{ass}

\begin{lem}\label{lem:bounds}
	Let Assumption~\ref{ass:finite} hold.
	Then, for any non-negative $t$,
	\begin{enumerate}[label=\alph*),itemindent=-1em]
		\item\label{it:lem:bounds:a} $f/f_{j} \leq \exp(t I^{j}(s_\tau))$, and
		\item\label{it:lem:bounds:b} $\exp(-2(1 + \Delta)t) (1-\Delta/\sqrt 2\cdot e^{(2+3\Delta)t} \exp(\Delta (e^{2t}-1)/2)) \leq f/f_{j}$ on $\Omega_{j}$.
	\end{enumerate}
\end{lem}

\begin{proof}[Proof of Lemma~\ref{lem:bounds}]
	In order to prove statement~\ref{it:lem:bounds:a}, set $I^{\vee j}(s) = I(s) - I^{j}(s)$ and $Z_{j} = \sup_{s \in \Sc} I^{\vee j}(s)$.
	Moreover, define $\tau_{j}$ as the first index $i$ such that $I^{\vee j}(s_i) = Z_{j}$.
	Then, $Z_{j}$ is $\Fs^{\vee j}$-measurable, and we have
	\begin{equation}\label{eq:lem:bounds}
	\exp(t(Z_{j} + X_{j} + \Delta)) \geq f \geq \exp(t Z_{j}) \cdot \exp(t I^{j}(s_{\tau_{j}})).
	\end{equation}
	The random variable $\tau_{j}$ is $\Fs^{\vee j}$-measurable which implies $\E^{\vee j} [ I^{j}(s_{\tau_{j}}) ] = 0$.
	Thus, by Jensen's inequality, we obtain from the second estimate in~\eqref{eq:lem:bounds} that
	\begin{equation*}
		f_{j} \geq \exp(t Z_{j}) \cdot \E^{\vee j}[ \exp(t I^{j}(s_{\tau_{j}})) ] \geq \exp(t Z_{j}) \geq \exp(t I^{\vee j}(s_\tau)),
	\end{equation*}
	and consequently $f_{j} \geq f \cdot \exp(-t I^{j}(s_\tau))$ which implies statement~\ref{it:lem:bounds:a}.
	
	For the proof of~\ref{it:lem:bounds:b}, we retain the notation introduced in the proof of statement~\ref{it:lem:bounds:a}.
	From the left-hand side inequality in~\eqref{eq:lem:bounds}, we obtain
	\begin{align*}
		f_{j} &\leq e^{t(Z_{j} + \Delta)} \cdot \E [e^{tX_{j}} \1_{\Omega_{j}} ] + e^{t(Z_{j} + \Delta)} \cdot \E [e^{tX_{j}} \1_{\Omega_{j}^\complement}] \\
		&\leq e^{t(Z_{j}+1+\Delta)} + e^{t(Z_{j} + \Delta)} \cdot \E [e^{2tX_{j}}]^{1/2} \P(\Omega_{j}^\complement)^{1/2}.
	\end{align*}
	Multiplication with $\1_{\Omega_{j}}$ on both sides, using the estimate $\P(\Omega_{j}^\complement)^{1/2} \leq \Delta/ \sqrt 2$ from Lemma~\ref{lem:prob}, and recalling the formula for the moment-generating function of a Poisson distributed random variable, yields
	\begin{align*}
		f_{j} \1_{\Omega_{j}} \leq e^{t(Z_{j} + 1 + \Delta)} \1_{\Omega_{j}} + e^{t(Z_{j} + \Delta)} \cdot \exp(\Delta (e^{2t}-1)/2) \cdot \Delta /\sqrt 2 \cdot \1_{\Omega_{j}},
	\end{align*}
	from which we conclude by exploiting the right-hand side inequality of~\eqref{eq:lem:bounds} and the definition of $\Omega_{j}$ that
	\begin{equation*}
		f_{j} \1_{\Omega_{j}} \leq f e^{2(1+\Delta)t} \1_{\Omega_{j}} + f e^{(1+2\Delta)t} \cdot \exp(\Delta (e^{2t}-1)/2) \cdot \Delta/\sqrt 2\cdot \1_{\Omega_{j}},
	\end{equation*}
	and hence by elementary transformations
	\begin{equation*}
		(1- f/f_{j}\cdot  e^{(1+2\Delta)t} \exp(\Delta (e^{2t}-1)/2) \cdot \Delta/\sqrt 2) \cdot \1_{\Omega_{j}} \leq f/f_{j}  \cdot e^{2(1+\Delta)t} \1_{\Omega_{j}}.
	\end{equation*}
	Now, by the statement of assertion~\ref{it:lem:bounds:a} and the definition of $\Omega_{j}$
	\begin{equation*}
		(1- e^{(2+3\Delta)t} \exp(\Delta (e^{2t}-1)/2) \cdot \Delta /\sqrt 2) \cdot \1_{\Omega_{j}} \leq f/f_{j} \cdot e^{2(1+\Delta)t} \cdot \1_{\Omega_{j}},
	\end{equation*}
	which yields the claim assertion after division by $e^{2(1+\Delta)t}$. 
\end{proof}

In the sequel, we put $c(t,\ell) = 1- e^{(2+3\Delta)t} \exp(\Delta (e^{2t}-1)/2) \cdot \Delta/\sqrt 2$.
Note that $c(t,\ell) \leq 1$ and, for any fixed non-negative $t$, $c(t, \ell) \to 1$ as $\ell \to \infty$.
In particular, $c(t, \ell) \in [1/2,1]$, for sufficiently large $\ell$, say $\ell\geq \ell_0 = \ell_0(t)$.
Under the validity of Assumption~\ref{ass:finite}, we consider for any $j \in \{ 1 ,\ldots,\ell\}$ the strictly positive and $\Fs^{\vee j}$-measurable random variables $h_{j}$ defined by
\begin{equation}\label{eq:def:h}
h_{j} = \sum_{i=1}^m \P^{\vee j} (\tau = i) \exp(tI^{\circ j}(s_i)) = \E^{\vee j}[ \exp ( tI^{\circ j}(s_\tau) ) ].
\underline{}\end{equation}
From now on, we denote by $C$ a numerical constant \emph{independent of $\ell$} (but certainly depending on the fixed value of $t$ considered) whose value may change depending on the context.
The following Lemma~\ref{lem:aux:bounds} provides some estimates which are used for the rest of this section.

\begin{lem}\label{lem:aux:bounds}
	Let Assumption~\ref{ass:finite} hold and let
	\begin{equation*}
		\eta(x)=1-\exp(-x)-e^{2(1+\Delta)t - \log c(t,\ell)}x
	\end{equation*}
	for $\ell \geq \ell_0$. Then, the estimate $\E [\bullet] \leq C$ holds true, where $\bullet$ can be replaced by any of the following random variables: \begin{enumerate}[label=\alph*), itemsep=0pt]
		\item\label{it:lem:aux:bounds:a} $h_{j}^4$,
		\item\label{it:lem:aux:bounds:b} $(f_{j}-f)^4$,
		\item\label{it:lem:aux:bounds:c} $(f \log(f/f_{j}))^4$,
		\item\label{it:lem:aux:bounds:d} $(f\eta(tI^{j}(s_\tau)))^4$,
		\item\label{it:lem:aux:bounds:e} $(I^{j}(s))^4$,
		\item\label{it:lem:aux:bounds:f} $(I^{j}(s))^2$,
		\item\label{it:lem:aux:bounds:g} $\exp(t I^{\vee j} (s_\tau))$, and 
		\item\label{it:lem:aux:bounds:h} $\exp(4tI^{\vee j}(s_\tau))$.
	\end{enumerate}
	Here, $I^{\vee j}$ is defined as in the proof of Lemma~\ref{lem:bounds}.
	The constant $C$ can be chosen independent of $j$, and in statements~\ref{it:lem:aux:bounds:d}--\ref{it:lem:aux:bounds:h}, it can, in addition, be chosen independent of $s$ and $s_\tau$, respectively.
\end{lem}

\begin{proof}
	Let us only mention that for the proof of statement~\ref{it:lem:aux:bounds:c}, it is useful to apply statement~\ref{it:lem:bounds:a} from Lemma~\ref{lem:bounds}.
	Then, all the estimates are easy to derive and we thus omit the proof.
\end{proof}

\begin{lem}\label{lem:h}
	Let Assumption~\ref{ass:finite} hold, and let $h_{j}$ be defined as in~\eqref{eq:def:h}. Then, for all $\ell \geq \ell_0$, we have
	\begin{equation*}
		\sum_{j=1}^{\ell} \E [ (f-h_{j}) \1_{\Omega_{j}} ] \leq e^{2(1+\Delta)t - \log c(t,\ell)} \E \left[ f \right] \log \E [ f ] + C \ell^{-1/2}.
	\end{equation*}
\end{lem}

\begin{proof}[Proof of Lemma~\ref{lem:h}]
	We begin the proof with the observation that
	\begin{equation}\label{eq:lem:h:dec}
	\E [(f-h_{j}) \1_{\Omega_{j}}] = \E [ f- h_{j}] + \E [(h_{j}-f) \1_{\Omega_{j}^\complement}] \leq \E [ f- h_{j}] + \E [h_{j} \1_{\Omega_{j}^\complement}]
	\end{equation}
	where the last estimate is due to the fact that $f$ is non-negative.
	Thanks to~\eqref{eq:def:h}, we obtain the decomposition
	\begin{align*}
		\E [f-h_{j}] &= \E [f (1-\exp(-tI^{j}(s_\tau)) - te^{2(1+\Delta)t - \log c(t,\ell)}I^{j}(s_\tau)]\\
		&\hspace{1em}+ t e^{2(1+\Delta)t- \log c(t,\ell)} \E [f  I^{j}(s_\tau) ]\\
		&= \E [f \eta(tI^{j}(s_\tau)) \1_{\Omega_{j}}] +  \E [f \eta(tI^{j}(s_\tau)) \1_{\Omega_{j}^\complement}]\\
		&\hspace{1em}+ t e^{2(1+\Lambda/\ell)t - \log c(t,\ell)} \E [f  I^{j}(s_\tau) ],
	\end{align*}
	where the function $\eta$ is defined in Lemma~\ref{lem:aux:bounds}.
	Note that $\eta$ is non-increasing on the interval $[-2(1+\Delta)t + \log c(t,\ell), \infty)$.
	This fact in combination with Lemma~\ref{lem:bounds} implies that
	\begin{equation*}
		\E [ f \eta(tI^{j}(s_\tau)) \1_{\Omega_{j}} ] \leq \E [ (f-f_{j} - e^{2(1+\Delta)t - \log c(t,\ell)} f \log(f/f_{j}) ) \1_{\Omega_{j}} ].
	\end{equation*}
	By the identities $\1_{\Omega_{j}} = 1- \1_{\Omega_{j}^\complement}$ and $\E [f- f_{j}]=0$, we thus obtain
	\begin{align*}
		\E [ f \eta(tI^{j}(s_\tau)) \1_{\Omega_{j}} ] &\leq\E [(f_{j}-f) \1_{\Omega_{j}^\complement}] +  e^{2(1+\Delta)t - \log c(t,\ell)} \E[ f \log(f/f_{j}) \1_{\Omega_{j}^\complement} ]\\
		&\hspace{+1em} - e^{2(1+\Delta)t - \log c(t,\ell)} \E [ f \log(f/f_{j}) ].\notag
	\end{align*}
	Using H\"older's inequality and Lemma~\ref{lem:aux:bounds}, we obtain the estimate
	\begin{equation*}
		\E [(f_{j}-f) \1_{\Omega_{j}^\complement}] \leq \E [ (f_{j}- f)^4 ]^{1/4} \cdot \P ( \Omega_{j}^\complement )^{3/4} \leq C  \ell^{-3/2},
	\end{equation*}
	and by the same argument $\E[ f \log(f/f_{j}) \1_{\Omega_{j}^\complement} ] \leq C \ell^{-3/2}$, $\E [f \eta(tI^{j}(s_\tau)) \1_{\Omega_{j}^\complement}] \leq C \ell^{-3/2}$, and $\E [h_{j} \1_{\Omega_{j}^\complement}] \leq C \ell^{-3/2}$.
	Putting these estimates into~\eqref{eq:lem:h:dec}, we obtain
	\begin{equation*}
		\E [(f-h_{j}) \1_{\Omega_{j}}] \leq e^{2(1+\Delta)t - \log c(t,\ell)} ( t\E [f  I^{j}(s_\tau) ]- \E [ f \log(f/f_{j}) ]  ) + C \ell^{-3/2},
	\end{equation*}
	and by summation over $j$,
	\begin{equation*}
		\sum_{j=1}^{\ell} \E [(f-h_{j}) \1_{\Omega_{j}}] \leq e^{2(1+\Delta)t - \log c(t,\ell)} (  t  \E [f Z ] -  \sum_{j=1}^\ell \E [ f \log(f/f_{j}) ] ) + C\ell^{-1/2}.
	\end{equation*}
	By application of Proposition~4.1 from~\cite{ledoux1996talagrands}, we have
	$$- \sum_{j=1}^\ell \E [ f \log (f/f_{j})] \leq - \E [ f \log f  ] + \E \left[ f \right]  \log \E [ f ],$$
	and thus
	\begin{equation*}
		\sum_{j=1}^{\ell} \E [(f-h_{j}) \1_{\Omega_{j}}] \leq 
		e^{2(1+\Delta)t-\log c(t,\ell)} \E \left[ f \right] \log \E [ f ] + C \ell^{-1/2}.
	\end{equation*}
\end{proof}
\begin{lem}\label{lem:r}
	Consider the function $r$ defined through $r(t,x) = x \log x + (1+t)(1-x)$. 
	Then, for any $s \in \Sc$ and $t \geq 0$,
	\begin{equation*}
		\E [  r((1+\Delta)t, \exp (t I^{j}(s) ) ) \, \1_{\Omega_{j}} ] \leq  C t^2  \ell^{-3/2} + {t^2} \E [ ( I^{j}(s) )^2 ] / 2.
	\end{equation*}
\end{lem}

\begin{proof}[Proof of Lemma~\ref{lem:r}]
	For fixed non-negative $t$ consider the functions $\eta, \delta$ defined through $\eta(x)= r((1+\Delta)t,e^{tx})= e^{tx}tx + (1+(1+\Delta)t) (1-e^{tx})$ and $\delta(x)=\eta(x) - x\eta'(0)-\frac{(tx)^2}{2}$, respectively.
	We have $\delta(0)=0$ and $\delta'(x)=t^2 (x-(1+\Delta)) (e^{tx}-1)$.
	Thus, the sign of $\delta'(x)$ coincides with the one of $x(x-(1+\Delta))$.
	This implies that $\delta(x) \leq \delta(0) = 0$ for all $x \leq 1+\Delta$, and hence $\eta(x) \leq x \eta'(0)+ (tx)^2/2$.
	Since the estimate $I^{j}(s) \leq 1+\Delta$ holds on $\Omega_{j}$, we obtain by the preceding arguments
	\begin{equation*}
		r((1+\Delta)t, e^{tI^{j}(s)}) \1_{\Omega_{j}} \leq ( - (1+\Delta)t^2 I^{j}(s) + (tI^{j}(s))^2/2) \1_{\Omega_{j}}.
	\end{equation*}
	Taking expectations on both sides yields
	\begin{align*}
		\E[ r((1+\Delta)t, \exp (t I^{j}(s) ) ) \1_{\Omega_{j}} ] &\leq \E [ ( - (1+\Delta)t^2 I^{j}(s) + (tI^{j}(s))^2/2 )  \1_{\Omega_{j}} ].
	\end{align*}
	Therefrom, by means of the relation $\1_{\Omega_{j}} \leq 1$, we obtain
	\begin{equation*}
			\E [ r((1+\Delta)t, \exp (t I^{j}(s) ) ) \1_{\Omega_{j}} ] \leq -(1+\Delta) t^2 \E [I^{j}(s) \1_{\Omega_{j}}] + t^2\E [ ( I^{j}(s) )^2 ]/2.
	\end{equation*}
	Finally, the decomposition $1 = \1_{\Omega_{j}} + \1_{\Omega_{j}^\complement}$, H\"older's inequality and Lemma~\ref{lem:aux:bounds} imply that
	\begin{equation*}
		\E [ r((1+\Delta)t, \exp (t I^{j}(s) ) ) \1_{\Omega_{j}} ] \leq  Ct^2 \ell^{-3/2} + t^2 \E [ ( I^{j}(s) )^2 ]/2,
	\end{equation*}
	(recall that $\E [I^{j}(s)]=0$ for all $s \in \Sc$) which finishes the proof.
\end{proof}

\begin{rem}
	There is a correspondence between some of the auxiliary results proved above and results appearing in~\cite{klein2005concentration}.
	Lemmata~3.1, 3.2, and~3.3 therein correspond to our Lemmata~\ref{lem:bounds}, \ref{lem:h}, and~\ref{lem:r}, respectively.
	Both, results and proofs turn out to be more intricate in our PPP setup. 
\end{rem}

\subsection{Proof of Theorem~\ref{thm:right}}\label{subs:proof}

	First note that it is sufficient to prove statements~\ref{thm:right:a}--\ref{thm:right:c} of Theorem~\ref{thm:right} for the case of finite $\Sc$.
	Based on this, the case of countable $\Sc$ follows using the monotone convergence theorem.
	Thus, we assume from now on without loss of generality that $\Sc = \{s_1,\ldots,s_m\}$, and the preceding results from Section~\ref{subsubs:notation:right} (which were mostly obtained under the validity of Assumption~\ref{ass:finite}) are available.
	For fixed $t$ and $\ell \geq \ell_0 = \ell_0(t)$ (here, $\ell_0(t)$ is defined as in the preceding subsection), let us represent the PPP $N$ as the superposition of $\ell$ i.i.d.\,PPPs $N_{j}$ with intensity measure $\Lambdatilde$ as in~\eqref{eq:inf:div}.
	Then, application of Proposition~4.1 from~\cite{ledoux1996talagrands} and the decomposition $\1_{\Omega}=\1_{\Omega_{j}} + \1_{\Omega_{j}^\complement}$ yield
	\begin{align}\label{eq:dec:bw:square}
		\E \left[ f \log f \right] - \E [ f ] \log \E [f ] &\leq \sum_{j=1}^{\ell} \E [ f \log(f/f_{j}) ] \notag\\
		&\hspace{-3em}= \underbrace{\sum_{j=1}^{\ell} \E [ f \log(f/f_{j}) \, \1_{\Omega_{j}} ]}_{=: \, \square} + \underbrace{\sum_{j=1}^{\ell} \E [ f \log(f/f_{j}) \, \1_{\Omega_{j}^\complement} ] }_{=: \, \blacksquare},
	\end{align}
	and we investigate the two terms separately.
	
	\noindent \emph{Examination of $\square$}:
	For $j \in \{1,\ldots,\ell\}$, consider the strictly positive random variables $g_{j}$ defined through
	\begin{equation*}
		g_{j} = \sum_{i=1}^m \P^{\vee j}(\tau = i) \exp\left(t I(s_i) \right). 
	\end{equation*}
	We have the elementary decomposition
	\begin{equation}\label{eq:dec:centred:g}
	\E [ f \log ( f/f_{j} ) \1_{\Omega_{j}} ] = \E [ g_{j} \log ( f/f_{j} )  \1_{\Omega_{j}} ] + \E [ ( f-g_{j} ) \log ( f/f_{j} ) \1_{\Omega_{j}} ]. 
	\end{equation}
	Note that $\E^{\vee j} [ f/f_{j} ] = 1$, and thus
	\begin{equation*}
		\E [  g_{j} \log ( f/f_{j} )  \1_{\Omega_{j}} ]  \leq \sup \{  \E [  g_{j} h \1_{\Omega_{j}} ]  : h \text{ is }\Fs\text{-measurable with }\E^{j} [  e^h ]  \leq 1 \} .
	\end{equation*}
	Due to the duality formula for the relative entropy (cf., for instance, \cite{ledoux1996talagrands}, p.~83 or  \cite{massart2007concentration}, Proposition~2.12), we obtain
	\begin{equation*}
		\E [ g_{j} \log ( f/f_{j} ) \1_{\Omega_{j}} ] \leq \E [  g_{j} \1_{\Omega_{j}} \log(g_{j} \1_{\Omega_{j}}) ] - \E [ g_{j} \1_{\Omega_{j}} \log \E^{\vee j}[g_{j} \1_{\Omega_{j}} ] ] .
	\end{equation*}
	Putting this estimate into~\eqref{eq:dec:centred:g} yields
	\begin{align*}
		\E [  f \log\left( f/f_{j} \right) \1_{\Omega_{j}} ]  &\leq \E [  g_{j} \1_{\Omega_{j}} \log(g_{j} \1_{\Omega_{j}}) ]  - \E [  g_{j} \1_{\Omega_{j}} \log \E^{\vee j}[g_{j} \1_{\Omega_{j}} ] ] \\
		&\hspace{1em}+ \E [ ( f-g_{j} ) \log ( f/f_{j} ) \1_{\Omega_{j}} ],
	\end{align*}
	and by summation over $j$ we obtain
	\begin{align}
		\square &\leq \sum_{j=1}^{\ell} \E [ g_{j} \1_{\Omega_{j}} \log(g_{j} \1_{\Omega_{j}})] - \sum_{j=1}^{\ell} \E [g_{j} \1_{\Omega_{j}} \log \E^{\vee j}[g_{j} \1_{\Omega_{j}} ]]\notag \\
		&\hspace{1em} + \sum_{j=1}^{\ell} \E [ ( f-g_j) \log ( f/f_{j} ) \, \1_{\Omega_{j}} ]. \label{eq:dec:centred:g:sum}
	\end{align}
	Lemma~\ref{lem:bounds}, combined with the facts that $f-g_{j} \geq 0$ and $t I^{j}(s_\tau) \1_{\Omega_{j}} \leq  (1+\Delta)t \1_{\Omega_{j}}$, implies
	\begin{equation}\label{eq:ex:l:centred:bounds}
		\E [ (f-g_{j}) \log(f/f_{j}) \1_{\Omega_{j}} ] \leq (1+\Delta)t \E [  (f-g_{j})   \1_{\Omega_{j}} ].
	\end{equation}
	For $j \in \{1,\ldots, \ell \}$, consider the positive and $\Fs^{\vee j}$-measurable random variables $h_{j}$ as defined in~\eqref{eq:def:h}.
	By the variational definition of relative entropy (see \cite{ledoux1996talagrands}, Equation~(1.5) or \cite{massart2007concentration}, Proposition~2.12), we obtain
	\begin{align*}
		&\E^{\vee j} [  g_{j} \1_{\Omega_{j}} \log(g_{j} \1_{\Omega_{j}}) ]  - \E^{\vee j} [ g_{j} \1_{\Omega_{j}} \log \E^{\vee j}[g_{j} \1_{\Omega_ {j}} ] ]\\
		&\hspace{+15em}\leq \E^{\vee j}[  ( g_{j} \log (g_{j} / h_{j}) - g_{j} + h_{j} ) \1_{\Omega_{j}} ].
	\end{align*}
	By taking expectations on both sides of the last estimate, and combining the result with~\eqref{eq:ex:l:centred:bounds} we obtain from~\eqref{eq:dec:centred:g:sum} that
	\begin{align*}
		\square &\leq \sum_{j=1}^{\ell} \E [ ( g_{j} \log (g_{j}/h_{j}) + (1+(1+\Delta) t)(h_{j} - g_{j}) ) \1_{\Omega_{j}} ] \notag\\
		&\hspace{1em}+  (1+\Delta)t \sum_{j=1}^{\ell} \E [ \left( f-h_{j} \right)  \1_{\Omega_{j}} ] \eqdef \square_1 + \square_2 \label{eq:dec:centred:box}.
	\end{align*}
	In order to bound $\square_1$ from above, introduce the function $r$ defined via
	\begin{equation*}\label{eq:def:centred:r}
		r(t,x)=x \log x + (1+t) (1-x).
	\end{equation*}
	By the definition of $g_{j}$ and $h_{j}$ we have
	\begin{equation*}
		g_{j} \log (g_{j}/h_{j}) + (1+(1+\Delta)t) (h_{j}- g_{j}) = h_{j} r ((1+\Delta)t, g_{j}/h_{j}),
	\end{equation*}
	and the convexity of $r$ with respect to $x$ yields
	\begin{equation*}
		h_{j} r((1+\Delta) t,g_{j}/h_{j}) \leq \sum_{i=1}^{m} \P^{\vee j}(\tau = i) \exp(t I^{\vee j}(s_i))r((1+\Delta) t,\exp(t I^{j}(s_i))).
	\end{equation*}
	Hence, multiplication with $\1_{\Omega_{j}}$ and application of the $\E^{\vee j}$ operator yield
	\begin{equation*}\label{eq:dec:square:1}
	\begin{split}
	\E^{\vee j} [ h_{j} r((1+\Delta)t,g_{j}/h_{j})  \1_{\Omega_{j}} ]&\\  &\hspace{-7em}\leq  \sum_{i=1}^{m} \P^{\vee j}(\tau = i) \exp(t I^{\vee j}(s_i)) \E [ r((1+\Delta)t, \exp (t I^{j}(s_i) ) )  \1_{\Omega_{j}} ].
	\end{split}
	\end{equation*}
	The expectation on the right-hand side can be bounded by Lemma~\ref{lem:r}, and we obtain
	\begin{align}
		\E^{\vee j} [ h_{j} r((1+\Delta)t,g_{j}/h_{j}) \1_{\Omega_j} ] &\leq  C t^2 \ell^{-3/2} \E^{\vee j} [\exp(t I^{\vee j}(s_\tau))]\notag \\
		&\hspace{-6em}+  {t^2} \E^{\vee j} \left[ \sum_{i=1}^{m} \1_{\{\tau=i \}} \exp(t I^{\vee j}(s_i)) \E [ ( I^{j}(s_i))^2 ] \right]/2.\label{eq:dec:square:2}
	\end{align}
	In order to further bound the second term on the right-hand side of the last estimate, we consider the decomposition
	\begin{align}
		\E^{\vee j} \left[ \sum_{i=1}^{m} \1_{\{\tau=i \}} \exp(t I^{\vee j}(s_i)) \E [( I^{j}(s_i))^2 ] \right]&\notag\\
		&\hspace{-10em}= \E^{\vee j} \left[ \sum_{i=1}^{m} \1_{\{\tau=i \}} \exp(t I^{j}(s_i)) \1_{\Omega_{j}} \E [ ( I^{j}(s_i))^2 ] \right] \notag \\
		&\hspace{-10em}+ \E^{\vee j} \left[ \sum_{i=1}^{m} \1_{\{\tau=i \}} \exp(t I^{\vee j}(s_i)) \1_{\Omega_{j}^\complement} \E [ ( I^{j}(s_i))^2 ] \right]\label{eq:dec:square:3},
	\end{align}
	and we bound the two terms on the right-hand side of \eqref{eq:dec:square:3} separately.
	In order to treat the first one, note that on $\Omega_{j}$ we have $\exp(tI^{\vee j}(s_i)) \leq \exp\left( 2t(1+\Delta) + tI(s_i)\right)$, from which we conclude that
	\begin{align}
		\E^{\vee j} \left[ \sum_{i=1}^{m} \1_{\{\tau=i \}} \exp(t I^{\vee j}(s_i)) \1_{\Omega_{j}} \E [ ( I^{j}(s_i))^2 ] \right]& \notag\\
		&\hspace{-13em}\leq e^{2(1+\Delta)t} \E^{\vee j}\left[ \sum_{i=1}^{m} \1_{\{\tau =i \} } \exp(tI(s_i)) \E [ ( I^{j}(s_i))^2 ] \right].\label{eq:dec:square:4}
	\end{align}
	For the second term on the right-hand side of~\eqref{eq:dec:square:3}, we have by Lemma~\ref{lem:aux:bounds} that
	\begin{equation*}
		\E^{\vee j} \left[ \sum_{i=1}^{m} \1_{\{\tau=i \}} \exp(t I^{\vee j}(s_i)) \1_{\Omega_{j}^\complement} \E [ ( I^{j}(s_i))^2 ] \right]  \leq C \E^{\vee j} [  \exp(t I^{j} (s_{\tau}) ) \1_{\Omega_{j}^\complement}],
	\end{equation*}
	and thus by putting this last estimate and~\eqref{eq:dec:square:4} into~\eqref{eq:dec:square:3} we obtain
	\begin{align*}
		\E^{\vee j} &\left[ \sum_{i=1}^{m} \1_{\{\tau=i \}} \exp(t I^{\vee j}(s_i)) \E [ ( I^{j}(s_i))^2 ] \right] \\
		&\hspace{+5em}\leq e^{2(1+\Delta)t} \E^{\vee j}\left[ \sum_{i=1}^{m} \1_{\{\tau =i \} } \exp(tI(s_i)) \E [ ( I^{j}(s_i))^2 ] \right]\\
		&\hspace{6em}+ C \E^{\vee j} [  \exp(t I^{\vee j} (s_{\tau}) ) \1_{\Omega_{j}^\complement}].
	\end{align*}
	By taking expectations on both sides of~\eqref{eq:dec:square:2} and summation over $j$, we obtain by means of the derived estimates in combination with Lemma~\ref{lem:aux:bounds} that
	\begin{align*}
		\square_1 &\leq Ct^2 \ell^{-1/2} + {t^2}e^{2(1+\Delta) t} \E \left[ \sum_{i=1}^{m} \1_{ \{ \tau = i \} } \exp(tI(s_i)) \sum_{j=1}^{\ell} \E [ ( I^{j}(s_i))^2 ] \right]/2\\
		&\hspace{1em}+ {Ct^2} \sum_{j=1}^{\ell} \E [ \exp(t I^{\vee j} (s_{\tau}) ) \1_{\Omega_{j}^\complement}]/2.
	\end{align*}
	Since
	$\sum_{j=1}^{\ell} \E [ ( I^{j}(s_i))^2 ] \leq V$
	and
	$\E [ \exp(t I^{\vee j} (s_{\tau}) ) \1_{\Omega_{j}^\complement}]  \leq C \ell^{-3/2}$ (the last estimate follows from H\"older's inequality and Lemma~\ref{lem:aux:bounds}),
	we obtain
	\begin{equation*}
		\square_1 \leq Ct^2 \ell^{-1/2} + {t^2} e^{2(1+\Delta) t}V \E [f]/2.
	\end{equation*}
	A suitable bound for $\square_2$ follows directly from Lemma~\ref{lem:h}.
	By combining the derived estimates for $\square_1$ and $\square_2$, we obtain
	\begin{equation}\label{eq:est:white}
	\begin{split}
	\square &\leq C(1+t^2) \ell^{-1/2} +{t^2} e^{2(1+\Delta) t}V \E [f]/2\\
	&\hspace{1em}+ (1+\Delta)t e^{2(1+\Delta)t-\log c(t,l)} \E \left[ f \right] \log \E [ f ].
	\end{split}
	\end{equation}
	
	\noindent \emph{Examination of $\blacksquare$}:
	By H\"older's inequality, Lemmata~\ref{lem:prob},~\ref{lem:bounds}, and~\ref{lem:aux:bounds}, we have
	\begin{equation}\label{eq:est:black}
	\blacksquare \leq \sum_{j=1}^{\ell} \E [(tf I^{j}(s_\tau))^4]^{1/4} \P (\Omega_{j}^\complement)^{3/4} \leq C \ell^{-1/2}.
	\end{equation}
	
	We now merge the examinations of the terms $\square$ \ and $\blacksquare$.
	More precisely, by combining~\eqref{eq:dec:bw:square} with~\eqref{eq:est:white} and~\eqref{eq:est:black} and letting $\ell$ tend towards infinity we obtain that
	\begin{equation*}
	tL'(t) - (te^{2t}+1)L(t) \leq t^2 e^{2t} (V/2).
	\end{equation*}
	Now, the rest of the proof follows in complete analogy to the one of Theorem~2.1 in~\cite{klein2005concentration}.
\qed

\section{Proof of Theorem~\ref{thm:left}}\label{s:app:left}

\subsection{Notation and preparatory results}\label{subsec:aux:left}

Apart from redefinitions in the sequel, we maintain the notation introduced in Section~\ref{subsubs:notation:right} for the proof of Theorem~\ref{thm:right}.
In particular, we use again the representation $N \stackrel{d}{=} \sum_{j=1}^{\ell} N_{j}$ of the PPP $N$ as the superposition of independent PPPs $N_{j}$ with intensity $\Lambdatilde$ and use the shorthand notations $\Lambdab = \Lambda(\X)$ and $\Delta =\Lambdab/\ell$.
Besides, we retain the definition $\Omega_{j} = \{X_{j} \leq 1 \}$ where $X_{j} = N_{j}(\X)$.
Let us further assume that Assumption~\ref{ass:finite} holds, that is, $\Sc=\{s_1,\ldots,s_m\}$ is finite.
Define now
\begin{equation*}
	L_i(t) = \sum_{j=1}^{\ell} \log \E [\exp(-t I^{j}(s_i))], \qquad i \in \{1,\ldots,m\}.
\end{equation*}

The corresponding \emph{exponentially compensated} empirical process is
$T_i(t) \defeq S_n(s_i) + t^{-1} L_i(t)$.
In addition to $Z$, let us define $Z_t \defeq \sup_{i \in \{ 1,\ldots,m\} } T_i(t)$ (for notational convenience, from now on we use the shorthand notation $\sup_i$/$\inf_i$ when the supremum/infimum over $i \in \{1,\ldots,m\}$ is taken) and redefine $f =f(t) = \exp(-t Z_t)$ and $f_{kj}=f_{kj}(t) = \E^{j}[f]$ (the $\sigma$-fields $\Fs^{j}$ are defined as in Section~\ref{s:app:right}).
Finally, we define $F(t) = \E [f]$ and $\Lc(t) = \log F(t)$.
The main strategy of the proof given in the next subsection is to derive a differential inequality for $\Lc$.
Let $\tau=\tau(t)$ denote the minimal value of $i \in \{1, \ldots, m\}$ such that $Z_t = T_i(t)$.
As before, $C$ denotes always some constant (whose value is independent from $\ell$) which might have different values in different contexts. 

\begin{lem}\label{lem:aux:bounds:left}
	Let Assumption~\ref{ass:finite} hold.
	Then, the estimate $\E [X] \leq C$ holds true, where $X$ can be replaced by any of the following random variables: \begin{enumerate}[label=\alph*), itemsep=0pt]
		\item\label{it:aux:bounds:left:a} $\exp(-4t I^{j}(s_\tau))$,
		\item $(f_{j}-f-\widetilde \psi_\ell(t) f \log(f_{j}/f))$,
		\item $(f(e^{\eta_{j}} - 1 - \widetilde \psi_\ell(t) \eta_{j} ))^4$,
		\item $((g_{j}-f)\log(f_{j}/f))^4$,
		\item\label{it:aux:bounds:left:e} $(I_{j}(s_i))^4 e^{-4tI_{j}(s_i)}$, and
		\item\label{it:aux:bounds:left:f} $g_j \log(g_j/\E^j[g_j])$.
	\end{enumerate}
	Here $g_{j}$, $\eta_{j}$ and $\widetilde \psi_\ell$ are defined in Lemma~\ref{lem:phi} and its proof, respectively.
	The constant $C$ can be chosen independent of $j$, and in statements \ref{it:aux:bounds:left:a} and \ref{it:aux:bounds:left:e}, it can in addition be chosen independently of $s_\tau$ and $s_i$, respectively.
\end{lem}

\begin{proof}
	As for the proof of Lemma~\ref{lem:aux:bounds}, all the estimates are easily derived and we thus omit the proof.
\end{proof}
%

\begin{lem}\label{lem:psi}
	Let Assumption~\ref{ass:finite} hold and $\psi_\ell(t) = \frac{1}{2} ( 1+e^{2(1+\Delta)t} )$.
	Set $\ell_{ji}(t) = \log \E [\exp(-t I^{j}(s_i))]$.
	Then, the following estimates hold almost surely.
	\begin{enumerate}[label=\alph*)]
		\item\label{it:lem:psi:a} $f_{j}/f \leq \exp(t I^{j}(s_\tau) + \ell_{j\tau})$, and
		\item\label{it:lem:psi:b} $\exp(t I^{j}(s_\tau) + \ell_{j\tau}) \leq \psi_\ell(t) \cdot (1 + \alpha_\ell) + \beta \cdot \ell^{-3/2}$ on $\Omega_{j}$ where $\alpha_\ell$ is a monotone sequence decreasing to $0$ as $\ell$ to $\infty$ and $\beta > 0$.
	\end{enumerate}
\end{lem}

\begin{proof}
	For $s \in \Sc$, define $I^{\circ j}(s) = I(s) - I^{j}(s)$ and
	\begin{equation*}
	Z^{j} \defeq \sup_{s \in \Sc} ( I^{\circ j}(s) + t^{-1} \log \E [ \exp (-t I^{\circ j}(s)) ] ). 
	\end{equation*}
	Let $\tau_{j}$ be the first index $i \in \{ 1,\ldots,m \}$ such that
	\begin{equation*}
		Z^{j} =  I^{j}(s_{i}) + t^{-1} \log \E [ \exp (-t I^{j}(s_{i})) ].
	\end{equation*}
	Then,
	$f \leq \exp(-t Z^{j}) \exp(-t I^{j}(s_{\tau_{j}}) - \ell_{j\tau_{j}}(t)),$
	and hence $\E^{j}[f]\leq \exp(-tZ^{j})$.
	By definition of $Z^{j}$, we have
	$\exp ( -tZ^{j} ) \leq f \cdot \exp(tI^{j} (s_{\tau}) + \ell_{j\tau}(t) ),$
	and Statement~\ref{it:lem:psi:a} follows.
	In order to proof statement~\ref{it:lem:psi:b}, first note that $\exp(t I^{j}(s_\tau)) \leq e^{(1+\Delta)t}$ on $\Omega_{j}$, and it remains to find an estimate for
	\begin{equation*}
		\exp(\ell_{j\tau}(t)) = \E [\exp(-t I^{j}(s_\tau))].
	\end{equation*}
	Consider the decomposition
	\begin{equation}\label{eq:dec:lem:psi}
		\E [\exp(-tI^{j}(s_\tau))] = \E [\exp(-tI^{j}(s_\tau)) \1_{\Omega_{j}}] + \E [\exp(-tI^{j}(s_\tau)) \1_{\Omega_{j}^\complement}].
	\end{equation}
	In order to bound the first term on the right-hand side of~\eqref{eq:dec:lem:psi}, note that $\E [\exp(-tI^{j}(s_\tau)) \1_{\Omega_{j}}] \leq \E [e^{tY}]$ with $Y=-I^{j}(s_\tau) \1_{\Omega_{j}}$.
	Note that
	\begin{align*}
		\vert \E Y \vert \leq \E \vert Y \vert \leq \E \vert I^j(s_\tau) \vert \leq \E X_j + \Lambdatilde(\X) = \frac{2\Lambdab}{\ell} \to 0 
	\end{align*}
	as $\ell$ tends to $\infty$.
	Thus, by the convexity of the exponential function, we have
	\begin{align}
		\E [e^{tY}] &\leq \frac{1+ \Delta - \E Y}{2(1+\Delta)} e^{-(1+\Delta)t} + \frac{\E Y + 1+ \Delta}{2(1+\Delta)} e^{(1+\Delta)t}\notag\\
		&= \frac 1 2 (  e^{-(1 + \Delta) t} + e^{(1+\Delta)t} )  (1 + o(1)). \label{eq:lem:psi:2}
	\end{align}
	The second term one the right-hand side of~\eqref{eq:dec:lem:psi} is bounded using H\"older's inequality, Lemmata~\ref{lem:prob} and~\ref{lem:aux:bounds:left} as follows:
	\begin{equation}\label{eq:lem:psi:1}
	\E [\exp(-tI^{j}(s_\tau)) \1_{\Omega_{j}^\complement}] \leq \E [\exp(-4tI^{j}(s_\tau)) ]^{1/4} \cdot \P (\Omega_{j}^\complement)^{3/4} \leq C \ell^{-3/2},
	\end{equation}
	and statement~\ref{it:lem:psi:b} follows from the combination of~\eqref{eq:lem:psi:2} and~\eqref{eq:lem:psi:1}.
\end{proof}

\begin{lem}\label{lem:phi}
	For $j \in \{ 1, \ldots, \ell \}$, define positive random variables $g_{j}$ via
	\begin{equation*}
		g_{j} = \sum_{i=1}^{m} \P^{j} (\tau = i) \exp (-t I(s_i) - L_i(t)).
	\end{equation*}
	Set $\varphi_\ell = \widetilde \psi_\ell \cdot \log \widetilde \psi_\ell$ where $\widetilde \psi_\ell = \psi_\ell \cdot (1+ \alpha_\ell) + \beta \ell^{-3/2}$ with $\psi_\ell$, $\alpha_\ell$, and $\beta$ defined as in Lemma~\ref{lem:psi}.
	For sufficiently large $\ell$, let $\theta_\ell$ be the unique positive solution of the equation $\phi_\ell(t) = 1$.
	Then, for any $t \in (0, \theta_\ell)$,
	\begin{align*}
		\sum_{j=1}^\ell \E [(g_{j} - f) \log(f_{j}/f)] \\
		&\hspace{-6em}\leq \frac{\phi_\ell(t)}{1 - \phi_\ell(t)} \left( \sum_{j=1}^{\ell} \E (g_{j} \log (g_{j}/\E^{j} [ g_{j} ]) ) - \E [f \log f] \right) + C \ell^{-1/2} .
	\end{align*}
\end{lem}

\begin{proof}
	Since $I^{\circ j}$ is $\Fs^{j}$-measurable, it is easy to verify that $$\E^{j}[g_{j}] = \E^{j} [f \exp(tI^{j}(s_{\tau}) + \ell_{j\tau}(t) )],$$
	and hence,
	\begin{equation*}
		\sum_{j=1}^{\ell} \E [g_{j} - f] = \sum_{j=1}^{\ell} \E [ f (\exp(tI^{j}(s_{\tau}) + \ell_{j\tau}(t)) -1 ) ].
	\end{equation*}
	Set $\eta_{j} = t I^{j}(s_{\tau}) + \ell_{j\tau}(t)$.
	Then,
	\begin{align}
		\sum_{j=1}^{\ell} \E [g_{j} - f] &= \sum_{j=1}^{\ell} \E [f(e^{\eta_{j}} - 1 - \widetilde \psi_\ell(t) \eta_{j} )] + \widetilde \psi_\ell(t) \E [f \sum_{j=1}^{\ell} \eta_{j}] \notag \\
		& = \sum_{j=1}^{\ell} \E [f(e^{\eta_{j}} - 1 - \widetilde \psi_\ell(t) \eta_{j} )] - \widetilde \psi_\ell(t) \E [f \log f],\label{lem:phi:1}
	\end{align}
	since $\sum_{j=1}^{\ell} \eta_{j} = - \log f$.
	Consider the first term on the right-hand side of~\eqref{lem:phi:1}. First, by H\"older's inequality and Lemma~\ref{lem:aux:bounds:left}
	\begin{equation*}
		 \sum_{j=1}^{\ell} \E [f(e^{\eta_{j}} - 1 - \widetilde \psi_\ell(t) \eta_{j} ) \1_{\Omega_{j}^\complement}] \leq C \ell^{-1/2}.
	\end{equation*}
	In order to bound $\sum_{j=1}^{\ell} \E [f(e^{\eta_{j}} - 1 - \widetilde \psi_\ell(t) \eta_{j} ) \1_{\Omega_{j}}]$ from above, note that the function $x \mapsto e^x - 1 -x \widetilde \psi_\ell(t)$ is non-increasing on the interval $(-\infty, \log \widetilde \psi_\ell(t)]$.
	Hence, we obtain by Lemma~\ref{lem:aux:bounds:left} that
	\begin{align*}
	\sum_{j=1}^{\ell} \E [f(e^{\eta_{j}} - 1 - \widetilde \psi_\ell \eta_{j} ) \1_{\Omega_{j}}] &\leq \widetilde \psi_\ell \sum_{j=1}^{\ell} \E [f \log(f/f_{j})]\\
		&\hspace{1em}- \sum_{j=1}^{\ell} \E [(f_{j}-f-\widetilde \psi_\ell f\log(f_{j}/f))\1_{\Omega_{j}^\complement}]\\
		&\leq \widetilde \psi_\ell \sum_{j=1}^{\ell} \E [f \log(f/f_{j})] + C \ell^{-1/2}.
	\end{align*}
	Putting the obtained estimates into~\eqref{lem:phi:1} yields
	\begin{equation*}
		\sum_{j=1}^{\ell} \E [g_{j} - f] \leq \widetilde \psi_\ell(t) \left( \sum_{j=1}^{\ell} \E [f \log(f/f_{j})] - \E[f \log f] \right) + C \ell^{-1/2}.
	\end{equation*}
	Using the same argument as in the proof of Theorem~\ref{thm:right} yields
	\begin{align}
		\sum_{j=1}^{\ell} \E [g_{j} - f]&\leq \widetilde \psi_\ell(t) (  \sum_{j=1}^{\ell} \E [g_{j} \log (g_{j} / \E^{j}[g_{j}]) \notag \\
		& +(g_{j} - f) \log(f_{j}/f)] - \E [f \log f] )  + C \ell^{-1/2}.\label{eq:lem:phi:2}
	\end{align}
	Now, in order to prove the claim assertion of the lemma, take note of the decomposition
	\begin{align}\label{eq:dec:left}
			\sum_{j=1}^{\ell} \E [ (g_{j} - f) \log (f_{j}/f) ] &= \sum_{j=1}^{\ell} \E [(g_{j} - f) \log (f_{j}/f) \1_{\Omega_{j}}] \notag \\
			&\hspace{1em}+  \sum_{j=1}^{\ell} \E [ (g_{j} - f) \log (f_{j}/f) \1_{\Omega_{j}^\complement}].
	\end{align}
	Using statement~\ref{it:lem:psi:b} of Lemma~\ref{lem:psi}, the estimate~\eqref{eq:lem:phi:2} and the definition of $\phi_\ell$, we can bound the first term as follows (note that $g_{j}-f \geq 0$):
	\begin{align*}
		\sum_{j=1}^{\ell} \E [(g_{j} - f) \log (f_{j}/f) \1_{\Omega_{j}}] &\leq \log \widetilde \psi_\ell(t) \sum_{j=1}^{\ell} \E [ g_{j} - f ]\\
		&\hspace{-13em}\leq \phi_\ell(t) \left( \sum_{j=1}^{\ell} [g_{j} \log (g_{j}/\E^{j} [g_{j}]) + (g_{j} -f) \log (f_{j}/f)] - \E [f \log f] \right) + C \ell^{-1/2}.
	\end{align*}
	The second summand on the right-hand side of~\eqref{eq:dec:left} can be bounded using H\"older's inequality, Lemma~\ref{lem:prob} and Lemma~\ref{lem:aux:bounds:left} once more:
	\begin{align*}
		\sum_{j=1}^{\ell} \E [ (g_{j} - f) \log (f_{j}/f) \1_{\Omega_\ell^\complement}] &\leq	\sum_{j=1}^{\ell} \E [ (g_{j} - f)^4 (\log(f_{j}/f))^4]^{1/4} \P(\Omega_\ell^\complement)^{3/4}\\
		&\leq C \ell^{-1/2}.
	\end{align*}
	Combining the bounds obtained for the two terms in~\eqref{eq:dec:left} implies the assertion of the lemma.
\end{proof}

\begin{rem}
	Both $\widetilde \psi_\ell(t)$ and $\phi_\ell(t)$ are non-increasing in $\ell$ and non-decreasing in $t$.
	Hence, the solution $\theta_\ell$ of the equation $\phi_\ell = 1$ (which exists for sufficiently large $\ell$) is non-decreasing in $\ell$ and the limit $\theta_\infty \defeq \lim_{\ell \to \infty} \theta_\ell$ satisfies $\theta_\infty \in [0.46, 0.47]$ (cf. p.~1075 in~\cite{klein2005concentration}).
	The approximate value of $\theta_\infty$ is of interest for the proof of Theorem~\ref{thm:left} which is done by considering different cases for the value of $t$ (cf.~\cite{klein2005concentration} for details).
\end{rem}

\begin{lem}\label{lem:Y}
	Let $Y$ be a random variable with values in $(-\infty, 1 + \Delta]$ and $\E [Y^2] < + \infty$.
	Then, for any positive $t$,
	\begin{equation*}
		\E [t Y e^{tY}] - \E [e^{tY}] \log \E[e^{tY}] \leq \frac{\E[Y^2]}{(1+\Delta)^2} ( 1+ ( (1+\Delta)t - 1 ) e^{(1+\Delta)t}  ).
	\end{equation*}
\end{lem}

\begin{proof}
	The proof follows completely along the lines of the one of Lemma~4.4 in~\cite{klein2005concentration}, and we thus omit it.
\end{proof}

\begin{rem}
	Again, there is a correspondence between some of the auxiliary results here and the ones used in~\cite{klein2005concentration}.
	Lemmata~\ref{lem:psi} and~\ref{lem:phi} are versions of Lemmata~4.2 and~4.3 in~\cite{klein2005concentration} tailored to our framework.
	As already mentioned above, Lemma~\ref{lem:Y} is exactly the same as Lemma~4.4 in~\cite{klein2005concentration} with $Y$ being replaced with $\frac{Y}{1+\Delta}$.
\end{rem}

\subsection{Proof of Theorem~\ref{thm:left}}

	The essential arguments of the proof follow along the proof of Theorem~1.2 in~\cite{klein2005concentration}.
	Since the random functions $T_i(t)$ are analytic in $t$, the random function $f(t)$ is continuous and piecewise analytic as a function in $t$.
	Its derivative with respect to $t$ satisfies
	\begin{equation*}
		f' = -(Z_t + t Z_t')f
	\end{equation*}
	where $t Z_t' = L_{\tau}'(t) -t^{-1} L_{\tau} (t)$.
	Thus, by the Fubini's theorem, we have
	\begin{equation*}
		F(t) = 1 - \int_0^t \E [(Z_u + u Z_u') f(u)] du.
	\end{equation*}
	Hence, $F$ is absolutely continuous with respect to the Lebesgue measure, with a.e. derivative in the sense of Lebesgue given by $F'(t) = - \E [(Z_t + tZ_t') f]$.
	Moreover, the function $\Lambda = \log F$ has the a.e.\ derivative $F'/F$. 
	As in the proof of Theorem~\ref{thm:right}, application of Proposition~4.1 from~\cite{ledoux1996talagrands} yields
	\begin{align}
	\E [ f \log f ] - \E [ f ] \log \E [ f ] &\leq \sum_{j=1}^{\ell}  \E [ g_{j} \log( g_{j}/\E^{j}[g_{j}])] \notag \\ &\hspace{5em}+ \sum_{j=1}^{\ell} \E [ (f- g_{j} ) \log( f/f_{j}) ]\label{eq:left:1}
	\end{align}
	for any positive integrable random variables $g_{j}$ such that $\E [ g_{j} \log [g_{j}]] < \infty$.
	On the other hand,
	\begin{equation}\label{eq:left:2}
	\E [ f \log f] - \E [f] \log \E [f] = \E [ t^2 Z_t' f] + t F'(t) - F(t) \log F(t) \quad \text{a.e.}
	\end{equation}
	Combining~\eqref{eq:left:1} and~\eqref{eq:left:2} yields
	\begin{align*}
		tF'(t) - F(t) \log F(t) &\leq - \E [t^2Z_t'f] + \sum_{j=1}^\ell \E [g_{j} \log (g_{j}/\E^{j}[g_{j}])]\\
		&\hspace{+1em}+ \sum_{j=1}^\ell \E [(g_{j} - f) \log (f_{j}/f)].
	\end{align*}
	We now specialize this estimate with the choice
	\begin{equation*}
		g_{j} = \sum_{i=1}^{m} \P^{j} (\tau = i) \exp (-tI(s_i) - L_i(t)),
	\end{equation*}
	which coincides with the definition of $g_{j}$ in Lemma~\ref{lem:phi}.
	Applying Lemma~\ref{lem:phi} and algebraic transformations yields
	\begin{align*}
		(1-\phi_\ell(t)) (tF'(t) - F \log F) &\leq \phi_\ell(t) \cdot \E [t^2 Z_t' f-f \log f]\\
		&\hspace{-1em}- \E [t^2 Z_t' f_t] + \sum_{j=1}^{\ell} \E [g_{j} \log (g_{j}/ \E^{j}[g_{j}] )] + C \ell^{-1/2},
	\end{align*}
	where $\phi_\ell$ is defined in Lemma~\ref{lem:phi}.
	Using the identity $\E[t^2 Z_t' f - f \log f] = -t F'$, we obtain
	\begin{equation}\label{eq:left:3}
	t F' - (1-\phi_\ell(t)) F \log F \leq  - \E [t^2 Z_t' f] + \sum_{j=1}^{\ell} \E [g_{j} \log (g_{j}/\E^{j}[g_{j}])] + C\ell^{-1/2}.
	\end{equation}
	Define $\omega_{j} = g_{j}/\E^{j}[g_{j}]$. Then,
	$\E^{j} [g_{j} \log(g_{j}/ \E^{j}[g_{j}] )] = \E^{j} [g_{j}] \cdot \E^{j}[\omega_{j} \log \omega_{j}]$.
	Using the convexity of $x \mapsto x \log x$, we conclude that
	\begin{equation*}
		\E^{j}[g_{j}] \omega_{j} \log \omega_{j} \leq \sum_{i=1}^{m} \P^{j}(\tau = i)(-t I^{j}(s_i) - \ell_{ji}(t)) \exp(-t I(s_i) - L_i(t)),
	\end{equation*}
	and by appling the $\E^{j}$ operator on both sides we obtain
	\begin{align*}
		\E^{j}[g_{j} \log (g_{j}/\E^{j}[g_{j}])] &\\
		&\hspace{-3em}\leq \sum_{i=1}^{m} \P^{j}(\tau=i) \exp(-t I^{\circ j}(s_i) - L_i(t) + \ell_{ji}(t) ) (t \ell_{ji}'(t) - \ell_{ji}(t))\\
		&\hspace{-3em}= \E^{j} \left[ \sum_{i=1}^{m} \1_{\{\tau = i\}} \exp(-t I^{\circ j}- L_i + \ell_{ji})(t\ell_{ji}'(t) - \ell_{ji}) \right]. 
	\end{align*}
	Thus, by taking expectations
	$$\E [g_{j} \log(g_{j}/\E^{j}[g_{j}])] \leq \E [ f\exp(t I^{j}(s_\tau) + \ell_{j\tau}) (t \ell_{j\tau}' - \ell_{j\tau}) ].$$
	By H\"older's inequality and Lemma~\ref{lem:aux:bounds:left}, we have $\E [g_{j} \log(g_{j}/\E^{j}[g_{j}]) \1_{\Omega_{j}^\complement}] \leq C \ell^{-3/2}$.
	In order to bound $\E [g_{j} \log(g_{j}/\E^{j}[g_{j}]) \1_{\Omega_{j}}]$, first note that the convexity of the functions together with the fact that $\ell_{ji}(0)=0$ implies $t \ell_{j\tau}' - \ell_{j\tau} \geq 0$.
	Thus, we can use Lemma~\ref{lem:psi} in order to obtain
	\begin{equation*}
		 \E [g_{j} \log(g_{j}/\E^{j}[g_{j}]) \1_{\Omega_{j}}] \leq \widetilde \psi_\ell(t) \cdot \E [ (t \ell_{j\tau}'(t) - \ell_{j\tau}(t)) f] + C\ell^{-1/2}.
	\end{equation*}
	By the identity $t^2 Z_t' =t L_\tau' - L_ \tau$, we get
	\begin{equation}\label{eq:left:4}
	- \E [t^2 Z_t' f] + \sum_{j=1}^\ell \E [g_{j} \log(g_{j}/\E^{j}[g_{j}])] \leq (\widetilde \psi_\ell(t) -1) \E [(t L_\tau' - L_\tau)f].
	\end{equation}
	In order to bound the expectation on the right-hand side of the last estimate, let us first note that $tL_\tau' - L_\tau \leq \sup_i (tL_i' - L_i)$.
	In order to bound $\sup_i (tL_i' - L_i)$, introduce (for fixed $i \in \{1,\ldots,m\}$) the event $\widetilde \Omega_{j}$ defined via
	\begin{equation*}
		\widetilde \Omega_{j} = \{ I^{j}(s_i) \geq -(1 + \Delta) \}.
	\end{equation*}
	Thanks to the boundedness of the functions $s \in \Sc$, we have $\Omega_{j} \subseteq \widetilde \Omega_{j}$, hence $\widetilde \Omega_{j}^\complement \subseteq \Omega_{j}^\complement$.
	Setting $Y_{j} = - I^{j}(s_i)$, we obtain
	\begin{align}\label{eq:ell}
			t \ell_{ji}'(t) - \ell_{ji}(t) &\leq t  \E [\exp(t Y_{j}) Y_{j}] - \E [e^{tY_{j}}] \log \E [e^{tY_{j}}]\notag\\
			&\hspace{-6em}\leq t  \E [\exp(t Y_{j}) Y_{j} \1_{\widetilde \Omega_{j}^\complement}] + t  \E [\exp(t Y_{j}) Y_{j} \1_{\widetilde \Omega_{j}}] - \E [ e^{tY_{j} \1_{\widetilde \Omega_{j}}} ] \log \E [ e^{tY_{j} \1_{\widetilde \Omega_{j}}} ].
	\end{align}
	The first term on the right-hand side of~\eqref{eq:ell} is bounded using Lemma~\ref{lem:aux:bounds:left}:
	\begin{equation*}
		t\E [\exp(t Y_{j}) Y_{j} \1_{\widetilde \Omega_{j}^\complement}] \leq C \P (\widetilde \Omega_{j}^\complement)^{3/4} \leq C  \P (\Omega_{j}^\complement)^{3/4} \leq C  \ell^{-3/2}.
	\end{equation*}
	The second and third term on the right-hand side of~\eqref{eq:ell} are bounded using Lemma~\ref{lem:Y} which yields
	\begin{align*}
		\E [t \exp(t Y_{j}) Y_{j} \1_{\widetilde \Omega_{j}}] - \E [ e^{tY_{j} \1_{\widetilde \Omega_{j}}} ] \log \E [ e^{tY_{j} \1_{\widetilde \Omega_{j}}} ] &\\&\hspace{-12em}= \E [t \exp(t Y_{j} \1_{\widetilde \Omega_{j}}) Y_{j} \1_{\widetilde \Omega_{j}}] - \E [ e^{tY_{j} \1_{\widetilde \Omega_{j}}} ] \log \E [ e^{tY_{j} \1_{\widetilde \Omega_{j}}} ]\\
		&\hspace{-12em}\leq \frac{\E [Y_{j}^2]}{\left( 1 + \Delta \right)^2} ( 1+ ((1+\Delta)t - 1) e^{(1+\Delta)t} ).
	\end{align*}
	Hence summing over all $j$ in~\eqref{eq:ell} yields
	\begin{equation*}
		tL_i' - L_i \leq C \ell^{-1/2} + \frac{V}{\left(1+ \Delta \right)^2} ( 1+ ((1+ \Delta)t - 1) e^{(1+\Delta)t} ),
	\end{equation*}
	and this estimate holds for all $i \in \{1, \ldots, m\}$.
	Putting the obtained estimates into~\eqref{eq:left:3} and~\eqref{eq:left:4} and letting $\ell$ tend to $\infty$, we obtain
	\begin{equation*}
		t F' - (1-\phi) F \log F \leq (\psi(t)-1 ) F V ( 1 + (t-1) e^{t} ),
	\end{equation*}
	where $\psi(t) = \frac{1}{2} (1+e^{2t})$ and $\phi = \psi \log \psi$.
	Division by $F$ yields
	\begin{equation*}
		t \Lc' - (1-\phi)\Lc \leq \frac{V}{2} (e^{2t} - 1)  ( 1 + (t-1) e^{t}).
	\end{equation*}
	This differential inequality for $\Lc$ coincides with equation (4.21) in~\cite{klein2005concentration} and the rest of the proof follows along the lines of the one given in that paper (Lemma~4.1 in~\cite{klein2005concentration} which is used for the proof translates without changes in the proof to our framework, whereas the purely analytical Lemmata~4.5 and~4.6 in~\cite{klein2005concentration} can be borrowed unchanged).
\qed
\section{A useful consequence of Corollary~\ref{COR:RIGHT}}\label{s:integrated}

In this subsection, we state and prove another concentration inequality which is the essential ingredient to prove the auxiliary Lemma~\ref{l:ex:conc} in Appendix~\ref{s:app:adap}.
As will become clear from the proof, it can be regarded as an integrated version of statement~\ref{cor:right:c} from Corollary~\ref{COR:RIGHT}.

\begin{prop}\label{prop:conc}
	Let $N_1,\ldots,N_n$ be independent PPPs on some Polish space $\X$ with finite intensity measures $\Lambda_1,\ldots,\Lambda_n$.
	Set
	\begin{equation*}
		\nu_n(r)=\frac{1}{n} \sum_{k=1}^{n} \left\lbrace \int_{\X} r(x)\dd N_k(x) - \int_\X r(x)\dd \Lambda_k(x) \right\rbrace
	\end{equation*}
	for $r$ contained in a countable class $\Rc$ of real-valued measurable functions.
	
	\noindent Then, for any $\epsilon > 0$, there exist constants $c_1,c_2=1/6, c_3$ such that
	\begin{align*}
		\E \left[ \left( \sup_{r \in \Rc} \vert \nu_n(r) \vert^2 - c(\epsilon)H^2 \right)_+  \right]&\\
		&\hspace{-8em}\leq c_1 \left\lbrace \frac{\upsilon}{n} \exp\left( -c_2\epsilon \frac{nH^2}{\upsilon}\right) + \frac{M_1^2}{C^2(\epsilon) n^2} \exp\left(-c_3 C(\epsilon) \sqrt{\epsilon} \frac{nH}{M_1} \right) \right\rbrace 
	\end{align*}
	where $C(\epsilon)=(\sqrt{1+\epsilon} -1) \wedge 1$, $c(\epsilon)=2(1+2\epsilon)$ and $M_1$, $H$ and $\upsilon$ are such that
	\begin{equation*}
		\sup_{r \in \Rc} \norm{r}_\infty \leq M_1, \, \E \left[ \sup_{r \in \Rc} \vert \nu_n(r) \vert \right] \leq H,  \text{ and } \sup_{r \in \Rc}  \var \left( \int_\X r(x)dN_k(x) \right) \leq \upsilon \, \forall k.
	\end{equation*}
\end{prop}

\begin{rem}
	Analogues of Proposition~\ref{prop:conc} have been used in the context of adaptive nonparametric estimation at various places, see, for instance,~\cite{comte2006penalized},~\cite{lacour2008adaptive} and~\cite{johannes2013adaptive}.
	The proof follows along the lines of the proof given in~\cite{chagny2013estimation} to a great extent and is thus only sketched.
\end{rem}

\begin{proof}[Proof of Proposition~\ref{prop:conc}]
	For $r \in \Rc$ and $k \in \{1,\ldots,n\}$ define functions $s_r^k: \X \to \R$ via
	\begin{equation*}
		s_r^k(x) = \frac{r(x)}{M_1}. 
	\end{equation*}
	Hence, for all $r \in \Rc$ and and $k \in \{ 1,\ldots, n \}$, we have $\vert s_r^k(x)\vert \leq 1$ and we can apply statement~\ref{cor:right:c} of Corollary~\ref{COR:RIGHT} for $\Sc = \{ ( s_r^1, \ldots, s_r^n) : r \in \Rc \}$ (the quantity $Z$ then corresponds to $\frac{n}{M_1} \sup_{r \in \Rc} \nu_n(r)$).
	Application of Theorem~\ref{COR:RIGHT} \ref{cor:right:c} yields for any $x > 0$ that
	\begin{equation*}
		\P \left( \frac{n}{M_1} \sup_{r \in \Rc}  \nu_n(r) \geq \frac{n}{M_1} \E \left[ \sup_{r \in \Rc} \nu_n(r) \right] + x \right) \leq \exp \left( - \frac{x^2}{2\upsilon + 3x} \right)
	\end{equation*}
	with $\upsilon = 2 \E Z + V_n$ where $V_n=\sup_{r \in \Rc} \var \left( S_n(s_r) \right)$ and $S_n$ is defined as in the statement of Corollary~\ref{COR:RIGHT}.
	Specializing with $x=ny/M_1$, we have
	\begin{align*}
		\P ( \sup_{r \in \Rc} \nu_n(r) \geq H + y ) &\leq \P ( \sup_{r \in \Rc} \nu_n(r) \geq \E [ \sup_{r \in \Rc} \nu_n(r) ] + y )\\
		&\leq \exp \left( - \frac{n^2 y^2}{2M_1^2 \upsilon + 3 M_1n y} \right).  
	\end{align*}
	Note that on the one hand we have $\E Z \leq nH/M_1$, and on the other hand $V_n \leq \frac{n\upsilon}{M_1^2}$,
	which in combination imply $\upsilon \leq 2nH/M_1 + n\upsilon/M_1^2$.
	We have
	\begin{equation*}
		\P ( \sup_{r \in \Rc} \nu_n(r) \geq H + y ) \leq \exp\left(- \frac{ny^2}{2(2M_1H+\upsilon) + 3 M_1 y} \right)  
	\end{equation*}	 
	which is used to obtain
	\begin{align*}
		\P ( \sup_{r \in \Rc} \vert  \nu_n(r)\vert  \geq H + y ) &\leq \P ( \sup_{r \in \Rc} \nu_n(r) \geq H + y ) + \P ( \sup_{r \in \Rc} - \nu_n(r) \geq H + y )\\
		&= \P ( \sup_{r \in \Rc} \nu_n(r) \geq H + y ) + \P ( \sup_{r \in \Rc} \nu_n(-r) \geq H + y )\\
		&\leq 2 \exp\left(- \frac{ny^2}{2(2M_1H+\upsilon) + 3 M_1 y} \right).
	\end{align*}
	Below, we will apply this estimate for $y = \mu + \eta H$. With this choice of $y$, following an argument from~\cite{birge1998minimum}, one can obtain the estimate
	\begin{equation*}
		\frac{y^2}{2(2M_1H+\upsilon) + 3 M_1 y} \geq \frac{1}{3} \left[ \frac{\mu^2}{2\upsilon} \wedge \frac{2(\eta \wedge 1)}{7} \frac{\mu}{M_1} \right],
	\end{equation*}
	which in turn implies that
	\begin{equation*}\label{eq:prob:exp}
	\P ( \sup_{r \in \Rc} \vert \nu_n(r)\vert  \geq \mu + (\eta+1) H ) \leq 2 \exp \left( -\frac{n}{3} \left\lbrace \frac{\mu^2}{2\upsilon} \wedge \frac{2(\eta \wedge 1)}{7} \frac{\mu}{M_1} \right\rbrace \right). 
	\end{equation*}
	The proof of the claim assertion follows now from
	\begin{align*}
		\E \left[ \left( \sup_{r \in \Rc} \vert \nu_n(r) \vert^2 - 2(1+2\epsilon)H^2 \right)_+ \right]\\ 
		&\hspace{-4em}= \int_0^\infty \P \left( \sup_{r \in \Rc} \abs{\nu_n(r)}^2 \geq 2(1+2\epsilon)H^2 +t \right)\dd t,
	\end{align*}
	and computation of the integral on the right-hand side (see~\cite{chagny2013estimation} for details).
\end{proof}

\section{Proof of Theorem~\ref{thm:lower}}\label{APP:LOWER}

	Let us define $\zeta=\min \{ 1/(\Gamma \eta), 16\delta/L \}$ with $\delta = 1/2 - 1/(2\sqrt 2)$ and for each $\theta=(\theta_j)_{0 \leq \abs{j} \leq k_n^*}\in \{\pm 1\}^{2k_n^* + 1}$ the function $\lambda_\theta$ through
	\begin{align*}
	\lambda_\theta &= \frac{L}{2} +  \theta_0 \left(\frac{L^2\zeta}{16n}\right)^{1/2}  + \left(\frac{L^2\zeta}{16n}\right)^{1/2} \sum_{1 \leq \abs{j} \leq k_n^*} \theta_j \phi_j\\
	&= \frac{L}{2} + \left(\frac{L^2\zeta}{16n}\right)^{1/2} \sum_{0 \leq \abs{j} \leq k_n^*} \theta_j \phi_j.
	\end{align*}
	\noindent Then, the calculation
	\begin{align*}
	\norm{\left(\frac{L^2\zeta}{16n}\right)^{1/2} \sum_{0 \leq \abs{j} \leq k_n^*} \theta_j \phi_j}_\infty &\leq \bigg(\frac{L^2 \zeta}{16n}\bigg)^{1/2} \sum_{0 \leq \abs{j} \leq k_n^*} \sqrt{2}\\
	&\leq \bigg(\frac{L^2\zeta}{8}\bigg)^{1/2} \bigg(\sum_{0 \leq \abs{j} \leq k_n^*} \gamma_j^{-2} \bigg)^{1/2} \bigg(\sum_{0 \leq \abs{j} \leq k_n^*} \frac{\gamma_j^2}{n} \bigg)^{1/2}\\
	&\leq \bigg(\frac{L^2\zeta \Gamma}{8} \bigg)^{1/2} \left( \gamma_{\knast}^2 \cdot \frac{2\knast +1 }{n} \right)^{1/2}\\
	&\leq \bigg(\frac{L^2\zeta \eta \Gamma}{8} \bigg)^{1/2} \leq L/\sqrt 8
	\end{align*}
	shows that $\lambda_\theta \geq L \delta$. In particular, $\lambda_\theta$ is non-negative for all $\theta \in \{ \pm 1 \}^{2\knast+1}$.
	Moreover $\norm{\lambda_\theta}_\gamma^2 \leq L^2$ holds for each $\theta \in \{\pm 1 \}^{2\knast+1}$ due to the estimate
	\begin{align*}
	\norm{\lambda_\theta}_\gamma^2 &= \left[ \left(\frac{L^2}{4}\right)^{1/2}  + \theta_0 \left( \frac{L^2\zeta}{16n} \right)^{1/2}  \right]^2  + \frac{L^2\zeta}{16} \, \sum_{1 \leq \abs{j} \leq k_n^*} \frac{\gamma_j^2}{n}\\
	&\leq \frac{L^2}{2} + \left( \frac{L^2\zeta}{8n} \right) + \frac{L^2\zeta}{16} \cdot \gamma_{k_n^*}^2 \sum_{1 \leq \abs{j} \leq k_n^*} \frac{1}{n}
	\\
	&\leq \frac{L^2}{2} + \frac{L^2\zeta}{8} \cdot \gamma_{\knast}^2 \cdot \frac{2\knast +1}{n} \leq L^2.
	\end{align*}
	This estimate and the non-negativity of $\lambda_\theta$ together imply $\lambda_\theta \in \Lambda$ for all $\theta \in \{ \pm 1\}^{2k_n^* + 1}$.
	Let $\P_\theta$ denote the joint distribution of the i.i.d. sample $N_1,\ldots,N_n$ when the true parameter is $\lambda_\theta$.
	Let $\P_\theta^{N_i}$ denote the corresponding one-dimensional marginal distributions and $\E_\theta$ the expectation with respect to $\P_\theta$. 
	From now on, let $\widetilde \lambda$ be an arbitrary estimator of $\lambda$.
	We denote by $\betatilde_j$ and $\betatheta_j$ the Fourier coefficients of $\lambdatilde$ and $\lambda^\theta$, respectively, where $\lambdatilde$ is an arbitrary but fixed estimator of $\lambda$.
	The key argument of the proof is the reduction scheme
	\begin{align}
	\sup_{\lambda \in \Lambda} \E [ \Vert \widetilde \lambda - \lambda \Vert^2] &\geq \sup_{\theta \in \{\pm 1\}^{2k_n^*+1}} \E_\theta [\Vert \widetilde \lambda - \lambda_\theta \Vert^2]\notag\\
	&\hspace{-6em}\geq \frac{1}{2^{2k_n^*+1}} \sum \limits_{\theta \in \{\pm 1\}^{2k_n^*+1}} \E_\theta [\Vert \widetilde \lambda - \lambda_\theta \Vert^2]\notag\\
	&\hspace{-6em}= \frac{1}{2^{2k_n^*+1}} \sum \limits_{\theta \in \{\pm 1\}^{2k_n^*+1}} \sum_{0 \leq \abs{j} \leq k_n^*} \ \E_\theta [(\betatilde_j - \betatheta_j)^2]\notag\\
	&\hspace{-6em}=\frac{1}{2^{2k_n^*+1}} \sum_{\theta \in \{\pm 1\}^{2k_n^*+1}} \sum_{0 \leq \abs{j} \leq k_n^*} \frac{\E_\theta [(\betatilde_j - \betatheta_j)^2] + \E_{\theta^{(j)}} [(\betatilde_j - \beta^{\theta^{(j)}}_j)^2]}{2},\label{eq:rs:l:n}
	\end{align}
	where for $\theta \in \{\pm 1\}^{2k_n^*+1}$ the element $\theta^{(j)} \in \{\pm 1\}^{2k_n^*+1}$ is defined by $\theta^{(j)}_k = \theta_k$ for $k \neq j$ and $\theta^{(j)}_j = -\theta_j$.
	Consider the Hellinger affinity defined as $\rho(\P_\theta, \P_{\theta^{(j)}}) = \int \sqrt{d\P_{\theta} d\P_{\theta^{(j)}}}$.
	For an arbitrary estimator $\widetilde \lambda$ of $\lambda$ we have
	\begin{align*}
	\rho(\P_\theta, \P_{\theta^{(j)}}) &\leq \int \frac{\vert \betatilde_j - \betatheta_j \vert}{\vert \betatheta_j - \beta^{\theta^{(j)}}_j \vert}\sqrt{d\P_{\theta} d\P_{\theta^{(j)}}} + \int \frac{\vert \betatilde_j - \beta^{\theta^{(j)}}_j \vert}{\vert \betatheta_j - \beta^{\theta^{(j)}}_j \vert} \sqrt{d\P_{\theta} d\P_{\theta^{(j)}}}\\
	& \leq \bigg(\int \frac{(\betatilde_j - \betatheta_j)^2}{(\betatheta_j - \beta^{\theta^{(j)}}_j)^2} d\P_\theta \bigg)^{1/2} + \bigg(\int \frac{(\betatilde_j - \beta^{\theta^{(j)}}_j)^2}{(\betatheta - \beta^{\theta^{(j)}}_j)^2} d\P_{\theta^{(j)}} \bigg)^{1/2},
	\end{align*}
	from which we conclude by means of the elementary inequality $(a+b)^2 \leq 2a^2 + 2b^2$ that
	\begin{equation*}
	\frac 1 2 ( \betatheta_j - \beta^{\theta^{(j)}}_j )^2 \rho^2(\P_\theta, \P_{\theta^{(j)}}) \leq \E_\theta [(\betatilde_j - \betatheta_j)^2] + \E_{\theta^{(j)}} [( \betatilde_j - \beta^{\theta^{(j)}}_j)^2].
	\end{equation*}
	Recall that the Hellinger distance between two probability measures $\P$ and $\Q$ is defined as $H(\P,\Q)=(\int [\sqrt{d\P}-\sqrt{d\Q} ]^2 )^{1/2}$.
	Using~Theorem 3.2.1 from~\cite{reiss1993course}, we obtain
	\begin{align*}\label{eq:hell:n}
	H^2(\P_\theta^{N_i}, \P^{N_i}_{\theta^{(j)}}) &= \int (\sqrt{\lambda_\theta} - \sqrt{\lambda_{\theta^{(j)}}} )^2 = \int \frac{\vert \lambda_\theta - \lambda_{\theta^{(j)}} \vert^2}{(\sqrt{\lambda_\theta} + \sqrt{\lambda_{\theta^{(j)}}})^2}\\
	&\leq \frac{1}{4\delta L} \ \norm{\lambda_\theta-\lambda_{\theta^{(j)}}}^2_{2}
	= \frac{\zeta L}{16\delta n}\leq \frac{1}{n}.
	\end{align*}
	Consequently, with Lemma~3.3.10~(i) from~\cite{reiss1989approximate} it holds
	\begin{equation*}\label{eq:crux}
	H^2(\P_\theta, \P_{\theta^{(j)}}) \leq \sum_{i=1}^n H^2(\P_\theta^{N_i}, \P_{\theta^{(j)}}^{N_i}) \leq 1.
	\end{equation*}
	Thus, the relation $\rho(\P_\theta, \P_{\theta^{(j)}})=1-H^2(\P_\theta, \P_{\theta^{(j)}})/2$ implies $\rho(\P_\theta, \P_{\theta^{(j)}}) \geq {1/2}$.
	Finally, putting the obtained estimates into the reduction scheme~\eqref{eq:rs:l:n} implies
	\begin{align*}
	\sup_{\lambda \in \Lambda} \E [\Vert \widetilde \lambda - \lambda \Vert^2]&\\
	&\hspace{-5em}\geq \frac{1}{2^{2k_n^*+1}}  \sum_{\theta \in \{\pm 1\}^{2k_n^*+1}} \sum_{0 \leq \abs{j} \leq k_n^*} \frac{\E_\theta [(\betatilde_j - \betatheta_j)^2] + \E_{\theta^{(j)}} [( \betatilde_j - \beta_j^{\theta^{(j)}} )^2]}{2} \\
	&\hspace{-5em}\geq \sum_{0 \leq \abs{j} \leq k_n^*} \frac{1}{16} (\betatheta_j - \beta_j^{\theta^{(j)}} )^2
	= \frac{\zeta L^2}{64}\sum_{0 \leq \abs{j} \leq k_n^*} \frac{1}{n} \geq \frac{\zeta L^2}{64\eta} \cdot \Psin,
	\end{align*}
	which finishes the proof of the theorem since $\widetilde \lambda$ was arbitrary.
\qed

\section{Proof of Theorem~\ref{thm:adap:est}}\label{s:app:adap}

\subsection{Proof of Theorem~\ref{thm:adap:est}}
	Let us introduce the event
	\begin{equation*}
		\Xi = \{ (\beta_0 \vee 1)/2 \leq \betahat_0 \vee 1 \leq 2 (\beta_0 \vee 1) \},
	\end{equation*}
	the definition of which is used to obtain the decomposition
	\begin{equation*}
		\E [\Vert \widehat \lambda_{\knhat} - \lambda \Vert^2] \leq \underbrace{\E [  \Vert \widehat \lambda_{\knhat} - \lambda \Vert^2 \1_\Xi ]}_{\eqqcolon \square} + \underbrace{\E [  \Vert \widehat \lambda_{\knhat} - \lambda \Vert^2 \1_{\Xi^\complement}]}_{\eqqcolon \blacksquare}.
	\end{equation*}
	We establish uniform upper bounds for both terms separately.
	
	\noindent\emph{Uniform upper bound for $\square$}:
	Since the equation $\Upsilon_n(t) = \Vert \widehat \lambda_n-t \Vert^2 - \Vert \widehat \lambda_n \Vert^2$ holds for all $t \in \L^2$, we obtain that
	$
	\argmin_{t \in \Sc_k} \Upsilon_n(t) = \widehat \lambda_k
	$ for all $k \in \{0,\ldots,n\}$ where $\Sc_k$ denotes the linear subspace of $\L^2$ spanned by the $\phi_j$ with $\vert j \vert \leq k$.
	This identity combined with the definition of $\knhat$ yields for all $k \in \{0, \ldots, n\}$ the inequality chain
	\begin{equation*}
		\Upsilon_n(\widehat \lambda_{\knhat}) + \pen_{\knhat} \leq \Upsilon_n(\widehat \lambda_k) + \pen_k \leq \Upsilon_n(\lambda_k) + \pen_k
	\end{equation*}
	where $\lambda_k = \sum_{0 \leq \vert j \vert \leq k} \beta_j \phi_j$ is the projection of $\lambda$ on the finite-dimensional space $\Sc_k$.
	Hence, using the definition of the contrast, we obtain
	\begin{equation*}
		\Vert \widehat \lambda_{\knhat} \Vert^2 \leq \Vert \lambda_k \Vert^2 + 2 \langle \widehat \lambda_n, \widehat \lambda_{\knhat} - \lambda_k \rangle + \pen_k - \pen_{\knhat}
	\end{equation*}
	for all $k \in \{0,\ldots,n\}$, from which we conclude by setting $\widehat \Theta_n = \widehat \lambda_n - \lambda_n$ that
	\begin{equation}\label{eq:1}
	\Vert \widehat \lambda_{\knhat} - \lambda \Vert^2 \leq \Vert \lambda - \lambda_k \Vert^2 + \pen_k - \pen_{\knhat} + 2 \langle \widehat \Theta_n, \widehat \lambda_{\knhat} - \lambda_{k} \rangle 
	\end{equation}
	for all $k \in \{0,\ldots,n\}$.
	Consider the set $\Bc_k = \{ \lambda \in \Sc_k : \norm{\lambda}^2 \leq 1 \}$.
	By means of the inequality $2uv \leq \tau u^2 + \tau^{-1}v^2$, we obtain for every $\tau > 0$ and $t \in \Sc_k$ that
	\begin{equation*}
		2 \vert  \langle h,t \rangle \vert \leq 2 \, \Vert t \Vert \sup_{t \in \Bc_k} \vert \langle h, t  \rangle \vert \leq \tau \norm{t}^2 + \tau^{-1} \sup_{t \in \Bc_k} \vert\langle h,t \rangle\vert^2.
	\end{equation*}
	Combining this estimate with the estimate~\eqref{eq:1}, we obtain (note that $\widehat \lambda_{\knhat}- \lambda_k \in \Sc_{k \vee \knhat}$)
	\begin{equation*}
		\Vert \widehat \lambda_{\knhat} - \lambda \Vert^2 \leq \Vert \lambda- \lambda_k \Vert^2 + \pen_k - \pen_{\knhat} + \tau \Vert \widehat \lambda_{\knhat} - \lambda_k \Vert^2 + \tau^{-1} \sup_{t \in \Bc_{k \vee \knhat}} \vert  \langle \widehat \Theta_n, t  \rangle \vert^2.
	\end{equation*}
	We have $\Vert \widehat \lambda_{\knhat} - \lambda_k \Vert^2 \leq 2 \Vert \widehat \lambda_{\knhat} - \lambda \Vert^2 + 2 \Vert \lambda_k - \lambda \Vert^2$ and $\Vert \lambda - \lambda_k \Vert^2 \leq L^2\gamma_k^{-2}$ for all $\lambda \in \Lambda$ thanks to Assumption~\ref{ass:seq}.
	Hence, specializing with $\tau = 1/4$ implies
	\begin{equation*}
		\Vert \widehat \lambda_{\knhat} - \lambda \Vert^2 \leq 3 L^2\gamma_k^{-2} + 2\pen_k -2\pen_{\knhat} + 8  \sup_{t \in \Bc_{k \vee \knhat}} \vert \langle \widehat \Theta_n,t \rangle \vert^2,
	\end{equation*}
	which is used to obtain
	\begin{align*}
		\Vert \widehat \lambda_{\knhat} -\lambda \Vert^2 &\leq 3L^2\gamma_k^{-2} + 8 \left( \sup_{t \in \Bc_{k \vee \knhat}} \vert \langle \widehat \Theta_n,t \rangle \vert^2 - \frac{3(\beta_0 \vee 1) \cdot (2(k \vee \knhat)+1)}{n} \right)_+\\
		&\hspace{1em} + \frac{24(\beta_0 \vee 1)\cdot (2(k \vee \knhat)+1)}{n} + 2\pen_k - 2\pen_{\knhat}.
	\end{align*}
	Note that we have $2(k \vee \knhat) +1 \leq 2k + 2\knhat + 2$.
	Thus, due to the definition of both the penalty and $\Xi$ we obtain
	\begin{equation*}
		\begin{split}
		\Vert \widehat \lambda_{\knhat} -\lambda \Vert^2 \, \1_\Xi &\leq 3L^2 \gamma_k^{-2} + 120 (L \vee 1) \cdot \frac{2k+1}{n}\\
		&\hspace{1em} + 8 \left( \sup_{t \in \Bc_{k \vee \knhat}} \vert \langle \widehat \Theta_n, t \rangle \vert^2 - \frac{3(\beta_0 \vee 1) \cdot (2(k \vee \knhat) + 1)}{n} \right)_+.
	\end{split}
	\end{equation*}
	Since the last estimate holds for all $k \in \{ 0,\ldots,n \}$ and $\lambda \in \Lambda$, we obtain
	\begin{align}\label{eq:upper:square}
	\E [ \Vert \widehat \lambda_{\knhat} - \lambda \Vert^2 \, \1_\Xi ] &\leq (3L^2 + 120 (L \vee 1)) \min_{0 \leq k \leq n} \max \left\lbrace \frac{1}{\gamma_k^2},\frac{2k+1}{n} \right\rbrace \notag \\
	&\hspace{-1em}+ 8 \sum_{k=0}^{n} \E \left[ \left( \sup_{t \in \Bc_{k}} \vert\langle \widehat \Theta_n, t \rangle \vert^2 - \frac{3 (\beta_0 \vee 1) (2k+1)}{n} \right)_+\right].
	\end{align}
	We now apply Lemma~\ref{l:ex:conc} from Subsection~\ref{subs:adap:aux} which using $\lambda \in \Lambda$ yields that
	\begin{align*}
		\E \left[ \left( \sup_{t \in \Bc_{k}} \vert\langle \widehat \Theta_n, t \rangle \vert^2 - \frac{3(\beta_0 \vee 1)(2k+1)}{n} \right)_+\right]&\\
		&\hspace{-18em}\leq K_1 \left[  \frac{\sqrt{2k+1}(L \vee 1)L}{n} \exp \left( -K_2\sqrt{\frac{2k+1}{L^2}} \right) + \frac{2k+1}{n^2} \exp \left( -K_3 \sqrt{n} \right) \right],
	\end{align*}
	where $K_1$, $K_2$ and $K_3$ are numerical constants independent of $n$.
	The estimate $2k+1 \leq 3n$ for $k \leq n$ yields
	\begin{align*}
		\sum_{k=0}^{n} \left[ \left( \sup_{t \in \Bc_{k}} \vert\langle \widehat \Theta_n, t \rangle \vert^2 - \frac{3(\beta_0 \vee 1)(2k+1)}{n} \right)_+\right] &\\
		&\hspace{-11em} \lesssim  \sum_{k=0}^{\infty}\frac{\sqrt{2k+1}}{n} \exp \left( - K_2 \sqrt{\frac{2k+1}{L^2}} \right)
		+ \exp(-K_3 \sqrt{n}).
	\end{align*}
	Note that we have $\sum_{k=0}^{\infty} \sqrt{2k+1} \exp \left(  -K_2 \sqrt{2k+1}/L \right)  \leq C < \infty$ for some numerical constant $C$.
	Since all the computations up to now hold uniformly for all $\lambda \in \Lambda$, plugging the derived estimates into~\eqref{eq:upper:square}, we obtain 
	$$\sup_{\lambda \in \Lambda} \E [ \Vert \widehat \lambda_\knhat - \lambda \Vert^2 \1_{\Xi} ] \lesssim \min_{0 \leq k \leq n} \max \left\lbrace \frac{1}{\gamma_k^2}, \frac{2k+1}{n}\right\rbrace + n^{-1} + \exp(-K_3 \sqrt{n}).$$
	
	\noindent \emph{Uniform upper bound for $\blacksquare$}: In order to derive an upper bound for $\blacksquare$, first recall the definition $\lambda_k = \sum_{0 \leq \vert j \vert \leq k} \beta_j \phi_j$ from above.
	We obtain the identity
	\begin{equation}\label{eq:blacksquare:dec}
	\E [ \Vert \widehat \lambda_{\knhat} - \lambda \Vert^2 \1_{\Xi^\complement} ] = \E [ \Vert \widehat \lambda_{\knhat} - \lambda_{\knhat} \Vert^2 \1_{\Xi^\complement} ] + \E [ \Vert \lambda - \lambda_{\knhat} \Vert^2 \1_{\Xi^\complement} ].
	\end{equation}
	Since $\Vert \lambda - \lambda_{\knhat} \Vert^2 \leq \Vert \lambda \Vert^2 \leq L^2$ due to Assumption~\ref{ass:seq}, the second term on the right-hand side of~\eqref{eq:blacksquare:dec} satisfies
	\begin{equation}\label{eq:blacksquare:1}
	\E [ \Vert \lambda - \lambda_{\knhat} \Vert^2 \1_{\Xi^\complement} ] \leq L^2 \P (\Xi^\complement) \lesssim  n^{-1},
	\end{equation}
	where the probability estimate for $\Xi^\complement$ is proved below.
	In order to bound the first term on the right-hand side of~\eqref{eq:blacksquare:dec}, first note that
	\begin{align*}
		\E [ \Vert \widehat \lambda_{\knhat} - \lambda_{\knhat} \Vert^2 \1_{\Xi^\complement} ]& \leq \sum_{0 \leq \vert j\vert \leq n} \E [ ( \betahat_j- \beta_j )^2 \, \1_{\Xi^\complement} ]\\
		&\leq \P(\Xi^\complement)^{1/2} \sum_{0 \leq ( j\vert \leq n} \E [ ( \betahat_j- \beta_j )^4]^{1/2}.
	\end{align*}
    Therefrom, by applying Theorem~2.10 from~\cite{petrov1995limit} (with $p=4$ in the statement of this theorem), we conclude
	\begin{equation*}
		\E [ \Vert \widehat \lambda_{\knhat} - \lambda_{\knhat} \Vert^2 \1_{\Xi^\complement} ] \lesssim \P \left( \Xi^\complement \right)^{1/2},
	\end{equation*}
	and it remains to find a suitable bound for $\P(\Xi^\complement)$. We have
	\begin{equation*}
		\P (\Xi^\complement) = \P ( \betahat_0 \vee 1 < (\beta_0 \vee 1)/2) + \P ( \betahat_0 \vee 1 > 2(\beta_0 \vee 1)),
	\end{equation*}
	and the probabilities on the right-hand side can be bounded by Chernoff bounds for Poisson distributed random variables (see~\cite{mitzenmacher2005probability}, Theorem~5.4).
	More precisely, by considering different cases, one can obtain the uniform bounds
	\begin{align*}
		&\P ( \betahat_0 \vee 1 < (\beta_0 \vee 1)/2 ) \leq \exp ( - 2\omega_1(1/2) n)  \quad \text{and}\\
		&\P ( \betahat_0 \vee 1 > 2 (\beta_0) \vee 1) \leq  \exp (-\omega_2(1/2) n)
	\end{align*}
	with $\omega_1(\eta)= 1-\eta + \eta \log \eta >0$ and $\omega_2(\eta) = 1-\eta^{-1} - \eta^{-1} \log \eta> 0$. 
	Hence, putting together the estimates derived so far, we obtain
	\begin{equation}\label{eq:blacksquare:2}
	\E [ \Vert \widehat \lambda_{\knhat} - \lambda_{\knhat} \Vert^2 \1_{\Xi^\complement} ] \lesssim n^{-1}.
	\end{equation}
	Putting the estimates~\eqref{eq:blacksquare:1} and~\eqref{eq:blacksquare:2} into~\eqref{eq:blacksquare:dec} and noting that all the obtained estimates hold uniformly for $\lambda \in \Lambda$, we obtain that
	\begin{equation*}\label{eq:blacksquare:final}
		\sup_{\lambda \in \Lambda} \E [ \Vert \widehat \lambda_\knhat - \lambda \Vert^2 \1_{\Xi^\complement}] \lesssim n^{-1}.
	\end{equation*}
	Combining the derived bounds for $\square$ \ and $\blacksquare$ \ implies the statement of the theorem.
\qed

\subsection{Auxiliary results}\label{subs:adap:aux}

The following lemma is a version of Lemma A4 in~\cite{johannes2013adaptive} adapted to our framework.
In that paper, a circular deconvolution model was considered and the same way Lemma~A4 in~\cite{johannes2013adaptive} is obtained from a variant of Proposition~\ref{prop:conc} (cf. Lemma~A3 in~\cite{johannes2013adaptive} or Lemma~1 in~\cite{comte2006penalized}), the key ingredient for the proof of Lemma~\ref{l:ex:conc} is Proposition~\ref{prop:conc}.

\begin{lem}\label{l:ex:conc}
	For all $k \in \{0,\ldots,n\}$, we have
	\begin{align*}
		\E \left[ \left( \sup_{t \in \Bc_{k}} \vert  \langle \widehat \Theta_n, t  \rangle \vert^2 -  \frac{3(\beta_0 \vee 1) (2k+1) }{n} \right)_+\right] &\\ &\hspace{-14em} \leq K_1 \left\lbrace  \frac{\sqrt{2k+1} \cdot (\beta_0 \vee 1) \norm{\lambda}}{n} \exp \left( - K_2 \cdot \frac{\sqrt{2k+1}}{\norm{\lambda}} \right) \right.\\
		&\left. + \frac{2k+1}{n^2}  \exp \left( -K_3 \sqrt{n} \right) \right\rbrace,
	\end{align*}
	with numerical constants $K_1$, $K_2$, and $K_3$.
\end{lem}

\begin{proof}
	 For $t \in \Sc_k$, we write $t = \sum_{j=-k}^{k} \tau_j \phi_j$.
	 Then, it is readily verified that $\langle \widehat \Theta_n, t \rangle = \frac 1 n \sum_{i=1}^n \{\int_0^1 t(x)\dd N_i(x) - \int_0^1 t(x) \lambda(x)\dd x$\}.
	 Hence, it remains to find constants $M_1$, $H$ and $\upsilon$ satisfying the preconditions of Proposition~\ref{prop:conc}.
	 
	 \noindent\emph{Condition concerning $M_1$:} We have
	 \begin{align*}
	 	\sup_{t \in \Bc_k} \norm{t}_\infty^2 = \sup_{t \in \Bc_k} \sup_{y \in [0,1)} \vert t(y) \vert^2 &\leq \sup_{t \in \Bc_k} \sup_{y \in [0,1)} \left(  \sum_{j=-k}^{k} \vert \tau_{j} \vert \vert \phi_j(y) \vert \right)^2 \\
	 	&\leq \sup_{t \in \Bc_k} \sup_{y \in [0,1)} \left( \sum_{j=-k}^{k} \vert \tau_{j} \vert^2 \right) \left( \sum_{j=-k}^{k} \phi_j^2(y) \right)\\
	 	&\leq 2k+1 \eqdef M_1^2.
	 \end{align*}
	 \noindent\emph{Condition concerning $H$:} We have
	 \begin{align*}
	 	\E [ \sup_{t \in \Bc_k} \vert \langle \widehat \Theta_n, t \rangle \vert^2 ]&\leq \sup_{t \in \Bc_k}  \left( \sum_{j=-k}^{k} \vert \tau_{j} \vert^2 \right)\\
	 	&\hspace{1em}\cdot \E \left[ \sum_{j=-k}^k \bigg\vert \frac{1}{n} \sum_{i=1}^{n} \left\lbrace \int_{0}^{1} \phi_j(x)[\dd N_i(x) -\dd \Lambda_i(x)] \right\rbrace  \bigg\vert^2 \right] \\
	 	&\leq \frac{1}{n} \sum_{j=-k}^{k} \var \left( \int_{0}^{1} \phi_j(x) \dd N_1(x) \right)\\
	 	&\leq \frac{1}{n} \sum_{j=-k}^{k} \int_0^1 \phi_j^2(x) \lambda(x) \dd x\\
	 	&\leq \frac{2k+1}{n} \cdot \beta_0,
	 \end{align*}
	 and it follows from Jensen's inequality that we can choose
	 \begin{equation*}
	 	H=\left( (\beta_0 \vee 1) \cdot (2k+1)/ n\right)^{1/2}.
	\end{equation*}
	 \noindent\emph{Condition concerning $\upsilon$:}
	 We have
	 \begin{equation}\label{eq:ups:var}
	 	\var \left(\int_0^1 t(x)\dd N_1(x) \right) = \int_0^1 \vert t(x) \vert^2 \lambda(x)\dd x.
	 \end{equation}
	 Define $\e_j(t)=\exp(2\pi i jt)$ and set $\langle \lambda \rangle_j = \int_0^1 \lambda(t) \e_j(-t)\dd t$ using which the identity $\lambda = \sum_{j \in \Z} \langle \lambda \rangle_j \e_j$ holds.
	 We have
	 \begin{align*}
	 	\vert t(x) \vert^2 &= \left \langle \sum_{i=-k}^{k} \langle t \rangle_{i} \e_i(x), \sum_{j=-k}^{k} \langle t \rangle_{j} \e_j(x) \right \rangle=\sum_{i=-k}^k \sum_{j=-k}^k \langle t \rangle_{i} \overline{\langle t \rangle}_{j} \e_i(x) \e_{-j}(x),
	 \end{align*}
	 and thus by means of~\eqref{eq:ups:var} that $\var \left(\int_0^1 t(x)\dd N_1(x) \right) =  \langle A \langle t \rangle, \langle t \rangle \rangle$, where for $t \in \Bc_k$ we denote by $\langle t \rangle$ the vector $\left( \langle t \rangle_{-k},\ldots,\langle t \rangle_k \right)$ and by $A$ positive semi-definite matrix $A = (\beta_{j-i} )_{j,i=-k,\ldots,k}$.
	 Hence,
	 \begin{align*}
	 	\sup_{t \in \Bc_k} \var \left(\int_0^1 t(x)\dd N_1(x) \right) &\leq \sup_{t \in \Bc_k} \langle A^{1/2} \langle t \rangle, A^{1/2} \langle t \rangle \rangle= \sup_{t \in \Bc_k} \Vert A^{1/2} t \Vert^2= \Vert A \Vert.
	 \end{align*}
	 In order to bound $\Vert A \Vert$, recall for an arbitrary matrix $B=(b_{ij})$ the definitions
	 \begin{equation*}
	 	\Vert B \Vert_1 \defeq \max_j \sum_i \vert b_{ij}\vert \qquad \text{and} \qquad \Vert B \Vert_\infty \defeq \max_i \sum_j \vert b_{ij}\vert.
	 \end{equation*}
	 Note that by the Cauchy-Schwarz inequality we have both $\Vert A\Vert_1\leq \sqrt{2k+1} \Vert \lambda\Vert$ and $\Vert A\Vert_\infty \leq \sqrt{2k+1} \Vert \lambda\Vert$ and hence by the formula $\Vert A \Vert \leq \sqrt{\Vert A\Vert_1 \cdot \Vert A\Vert_\infty}$ (see Corollary~2.3.2 in~\cite{golub1996matrix}) we obtain $\Vert A \Vert \leq \sqrt{2k+1} \cdot \Vert \lambda \Vert$. Thus, we can choose $\upsilon \defeq \sqrt{2k+1} \cdot \Vert \lambda \Vert \cdot (\beta_0 \vee 1)$.
	 
	  The result of the Lemma follows now directly from Proposition~\ref{prop:conc} with $\epsilon = \frac{1}{4}$. 
\end{proof}

%

\printbibliography

\end{document}